\documentclass[12pt]{article}
\usepackage{mathrsfs}
\usepackage{amsthm}
\usepackage{amssymb}
\usepackage{latexsym}
\usepackage{amsmath,amsfonts}
\usepackage{mathrsfs}
\usepackage{cases}
\usepackage{latexsym,bm}
\usepackage{indentfirst}
\usepackage{color}
\usepackage{ifpdf}
\usepackage{graphicx}
\usepackage{psfrag}
\usepackage[
pdfauthor={CSY},
pdftitle={Hamiltonicity of edge-chromatic critical graphs},
pdfstartview=XYZ,
bookmarks=true,
colorlinks=true,
linkcolor=blue,
urlcolor=blue,
citecolor=blue,
bookmarks=true,
linktocpage=true,
hyperindex=true
]{hyperref}

\usepackage{graphicx}
\usepackage{color}

\title{ Hamiltonicity of edge-chromatic critical graphs }
\author{Yan Cao$^a$,
 Guantao Chen$^a$, 
 Suyun Jiang$^b$\thanks{Partially supported by NSFC of China (Nos. 11671232, 11571096, 61373019, 11671186).}, 
 Huiqing Liu$^{c*}$, 
 Fuliang Lu$^{d*}$
 \unskip\\[.5em]
{\small $^a$  Department of Mathematics and Statistics, Georgia State University, Atlanta, GA 30303}\\
{\small $^b$  School of Mathematics, Shandong University, Jinan, Shandong 250100}\\
{\small $^c$  Hubei Key Laboratory of Applied Mathematics, Faculty of Mathematics and Statistic,}\\
{\small Hubei University, Wuhan, Hubei 430062}\\
{\small $^d$ School of  Mathematics and Statistics, Linyi University, Linyi, Shandong 276000}\\
}
\date{}

\newtheorem{lem}{Lemma}
\newtheorem{thm}{Theorem}

\newtheorem{conj}{Conjecture}

\newtheorem{cla}{Claim}[section]

\newcommand{\D}{\Delta}

\newcommand{\phiv}{\varphi}
\newcommand{\phibar}{\bar{\varphi}}

\begin{document}
\newcommand{\udots}{\mathinner{\mskip1mu\raise1pt\vbox{\kern7pt\hbox{.}}
\mskip2mu\raise4pt\hbox{.}\mskip2mu\raise7pt\hbox{.}\mskip1mu}}
\maketitle

\begin{abstract}
Given a graph  $G$, denote by  $\D$ and $\chi^\prime$  the maximum degree and
the  chromatic index of $G$, respectively.   A simple graph $G$ is called  {\it edge-$\Delta$-critical}   if $\chi^\prime(G)=\Delta+1$  and $\chi^\prime(H)\le\Delta$  for every proper subgraph $H$ of $G$.
We proved that every edge chromatic critical graph of order $n$ with maximum degree at least $\frac{2n}{3}+12$ is Hamiltonian.
\\
\par {\small {\it Keywords: }\ \  edge-$k$-coloring; edge-critical graphs; Hamiltonicity}
\end{abstract}
\vskip 0.2in \baselineskip 0.1in
\section{ Introduction }

All graphs in this paper are simple graphs, that is, graphs with finite number of vertices without loops or parallel edges. Let $G$ be a graph with maximum degree $\D$ and minimum degree $\delta(G)$. Denote by $V(G)$ and $E(G)$ the vertex-set and edge-set of a graph $G$, respectively.  An edge-$k$-coloring of a graph $G$ is a mapping $\varphi:E(G)\rightarrow \{1,2,\cdots,k\}$ such that  $\phiv(e)\neq\phiv(f)$ for any two adjacent edges $e$ and $f$.  The codomain $\{1, 2, \cdots, k\}$ is called the color  set of $\phiv$.  Denote by $\mathcal{C}^k(G)$ the set of all edge-$k$-colorings of $G$. The chromatic index $\chi^\prime(G)$ is the least integer $k\geq 0$ such
that $\mathcal{C}^k(G)\neq \emptyset$.  We call $G$ {\it class} I if $\chi^\prime(G)=\Delta$. Otherwise, Vizing \cite{viz1} gives
$\chi^\prime(G)=\Delta+1$ and $G$ is said to be of {\it class} II.
An edge $e$  is called {\it critical} if
$\chi^\prime(G-e)<\chi^\prime(G)$, where $G-e$ is the subgraph obtained from $G$ by removing the edge $e$.
A graph $G$ is called {\it edge-$\D$-critical} if  $\chi^\prime(G) = \D+1$ and $\chi^\prime(H) \le \D$ holds for any proper subgraph $H$ of $G$.  Clearly, if $G$ is edge-$\D$-critical, then $G$ is connected and $\chi^\prime(G-e) = \D$ for any $e\in E(G)$.

In 1965, Vizing \cite{viz2} proposed the following conjecture about the structure of edge-$\Delta$-critical graphs.

\begin{conj}\label{con1} \emph{[Vizing \cite{viz2}]}.
Every edge-$\Delta$-critical graph with chromatic index at least 3 contains a 2-factor.
\end{conj}

In 1968, Vizing \cite{viz3} proposed a weaker conjecture on the independence number of edge-$\Delta$-critical graphs as follows.

\begin{conj}\label{con2} \emph{[Vizing \cite{viz3}]}.
For every edge-$\Delta$-critical graph of order $n,$ $\alpha(G)\leq \frac{n}{2}.$
\end{conj}

Vizing's independence number conjecture was verified by Luo and Zhao \cite{lz4} for edge-$\Delta$-critical graphs of order $n$ with $\Delta\geq \frac{n}{2},$ and by Gr\"{u}newald and Steffen \cite{gs} for edge-$\Delta$-critical graphs with many edges, including all overfull graphs.

Chen and Shan \cite{cs} verified Vizing's 2-factor conjecture for edge-$\Delta$-critical graphs of order $n$ with $\Delta\geq \frac{n}{2}.$ Obviously, if a graph is Hamiltonian, then it contains a 2-factor. Luo and Zhao \cite{lz} proved that an edge-$\Delta$-critical graph $G$ of order $n$ with $\Delta\geq \frac{6n}{7}$ is Hamiltonian. Furthermore, Luo, Miao and Zhao \cite{lmz2} showed that an edge-$\Delta$-critical graph $G$ of order $n$ with $\Delta\geq \frac{4n}{5}$ is Hamiltonian. Recently, Chen, Chen and Zhao \cite{cchen} showed that an edge-$\Delta$-critical graph $G$ of order $n$ with $\Delta\geq \frac{3n}{4}$ is Hamiltonian.

In this paper, we give the following result about the hamiltonicity of $\Delta$-critical graphs.

\begin{thm}\label{thml}
If $G$ is an edge-$\Delta$-critical graph of order $n$ with $\Delta\geq \frac{2n}{3}+12,$ then $G$ is Hamiltonian.
\end{thm}

It would be nice to know the minimum number $\alpha$ ($0<\alpha<1$) such that every edge-$\Delta$-critical graph of order $n$ with $\Delta\geq\alpha n$ is Hamiltonian.
Our main techniques applied to prove Theorem \ref{thml} are the following: (1) extending Woodall's Lemma ($q=2\Delta-d(x)-d(y)+2$, see Lemma \ref{p}) to an arbitrary $q$ with $q\leq \Delta-10$ (see Lemma \ref{gen vizingfan}); (2) extending Woodall's Lemma (consider the neighbor of a vertex $x$, see Lemma \ref{p}) to Lemma \ref{gen broom} (consider the neighbor of two adjacent vertices).

We will give a few technic lemmas in Section 2 and give the proof of Theorem \ref{thml} in Section 3. Due to the length of the proofs of Lemmas 4 and 5, we will prove Lemmas 4 and 5 in Section 4.
Let $G$ be a graph and $x$ be a vertex of $G$. Denote by
$N(x)$ and $d(x)$ the neighborhood and degree of $x$, respectively. For any nonnegative integer $k$, we call a vertex $x$ a $k$-vertex if $d(x) = k$, a $(<k)$-vertex if $d(x)< k$, and $(>k)$-vertex if $d(x)> k$. Correspondingly, we call a neighbor $y$ of $x$ a $k$-neighbor if $d(y) = k$, etc..
Denote by $V_{\ge k}(G)$ the subset of $V(G)$ of vertices with degree greater than or equal to $k$.
Let $k$ be a positive integer such that $\mathcal{C}^k(G-e) \ne \emptyset$,    $\varphi\in\mathcal{C}^k(G-e)$ and $v \in V(G)$.   Let $\varphi(v) = \{\varphi(e) \,:\,   e \mbox { is incident with }~ v \}$ and $\bar{\varphi}(v) = \{ 1, \cdots, k\}\setminus\varphi(v)$.
We call $\varphi(v)$ the set of colors seen by  $v$ and $\bar{\varphi}(v)$ the set of colors missing at $v$.
A set $X\subseteq V(G)$
is called {\it elementary} with respect to $\varphi$  if $\bar{\varphi}(u)\cap \bar{\varphi}(v)=\emptyset$ for every two distinct
vertices $u,v \in X$. For any color $\alpha$, let $E_{\alpha}$ denote the set of edges assigned color $\alpha$. Clearly, $E_{\alpha}$ is a matching of $G$.  For any two colors $\alpha$ and $\beta$, the components  of induced by edges in $E_{\alpha}\cup E_{\beta}$, named $(\alpha, \beta)$-chains,   are even cycles and paths with alternating color $\alpha$ and $\beta$.
For a vertex $v$ of $G$, we denote by
$P_v(\alpha, \beta, \varphi)$ the unique $(\alpha, \beta)$-chain that contains the
vertex $v$.   Let $\varphi/ P_v(\alpha, \beta, \varphi)$ denote the edge-$k$-coloring obtain from $\phiv$ by switching colors $\alpha$ and $\beta$ on the edges on $P_v(\alpha, \beta, \phiv)$.

\section{ Lemmas }

Let $q$ be a positive number, $G$ be an edge-$\D$-critical graph and $x\in V(G)$. For each $y\in N(x)$, let $\sigma_q(x, y) =  |\{ z\in N(y)\setminus \{x\} \ : \ d(z) \ge q\}|$, the number of neighbors of $y$ (except $x$) with degree at least $q$.  Vizing studied the case $q=\D$ and obtained the following result.

\begin{lem} {\em [Vizing's Adjacency Lemma \cite{viz2}]}\label{vizingaj}
Let $G$ be  an edge-$\Delta$-critical graph.  Then   $\sigma_{\D} (x, y) \ge \D -d(x) +1$ holds for every  $xy\in E(G)$.
\end{lem}

Woodall~\cite{w1} studied  $\sigma_q(x,y)$ for the case $q=2\D -d(x) -d(y) +2$ and obtained the following two results.   For convention, we let $\sigma(x,y) = \sigma_q(x, y)$ when $q=2\D -d(x) -d(y) +2$.
\begin{lem}\label{lemma2.4}{\em [Woodall~\cite{w1}]} Let $xy$ be an edge in an  edge-$\Delta$-critical graph $G$. Then there are at least $\Delta-\sigma(x,y)\geq \Delta-d(y)+1$ vertices $z\in N(x)\setminus\{y\}$ such that $\sigma(x,z)\geq 2\Delta-d(x)-\sigma(x,y).$
\end{lem}
Furthermore, Woodall defined the following two parameters.
\begin{eqnarray*}
p_{min}(x)& := & \min_{y\in N(x)}\sigma(x,y)-\Delta+d(x)-1 \quad  \mbox{ and }\\
p (x) & := & \min \{\ p_{min}(x), \left\lfloor\frac{d(x)}{2}\right\rfloor-1\ \}.
\end{eqnarray*}
Clearly, $p(x) < d(x)/2 -1$. As a corollary, the following lemma shows that there are about $d(x)/2$ neighbors $y$ of $x$ such that $\sigma(x, y)$ is at least $\D/2$.

\begin{lem}\label{p} {\em [Woodall~\cite{w1}]} Every vertex $x$ in an edge-$\Delta$-critical graph has
 at least $d(x)-p(x)-1$ neighbors $y$ for which $\sigma(x,y)\geq \Delta-p(x)-1$.
\end{lem}

In our proof, we only need one neighbor $y$ of $x$ such that $\sigma_q(x,y)$ is large. By allowing $q$ to take various values, roughly speaking, we can count the number of edges more flexibility. The only drawback is that the lower bound $\sigma_q(x,y)$ in Lemmas \ref{gen vizingfan} and  \ref{gen broom} are smaller, but not too much. It is much clearly in Lemmas \ref{gen vizingfan} and \ref{gen broom}.

\begin{lem}\label{gen vizingfan}
Let $xy$ be an edge in an edge-$\Delta$-critical graph $G$ and $q$  a positive number. If $d(x)<\frac{\Delta}{2}$ and $q\leq \Delta-10$, then
there exists a vertex $z\in N(x)\setminus\{y\}$ such that $\sigma_q(x,y)+\sigma_q(x,z)>2\Delta-d(x)-\frac{6d(x)-8}{\Delta-q}-\frac{8(d(x)-2)}{(\Delta-q)^2}.$
\end{lem}

\begin{lem}\label{gen broom}
Let $x_1x_2$ be an edge in an  edge-$\Delta$-critical graph $G$ and $q$  a positive number. If $d(x_1)+d(x_2)<\frac{3}{2}\D$, $q\leq \Delta-10$ and $\delta(G)\geq \frac{\Delta}{2}$, then
there exists a pair of vertices $\{z,y\}$ with $z\in N(x_1)\setminus \{x_2\}$ and $y\in N(x_2)\setminus \{x_1,z\}$ such that
$\sigma_q(x_1,z)+\sigma_q(x_2,y)>
3\Delta -d(x_1)-d(x_2)-\frac{6(d(x_1)+d(x_2)-\Delta)-4}{\Delta-q}-\frac{8(d(x_1)+d(x_2)-\Delta-2)}{(\Delta-q)^2}$.
\end{lem}

Our approaches are inspired by the recent development of Tashkinov tree technique for multigraphs.
Let $G$ be a multigraphs without loops, $e_1=y_0y_1\in E(G)$ and $\phiv\in C^k(G-e_1)$.
A Tashkinov tree $T$ with respect to $G, e, \phiv$ is an alternating sequence $T=(y_0,e_1,y_1,\cdots,e_p,y_p)$ with $p\geq 1$ consisting of edges $e_1, e_2, \cdots,e_p$ and vertices $y_0,y_1\cdots,y_p$ such that the following two conditions hold.
\begin{itemize}
\item The edges $e_1, e_2, \cdots,e_p$ are distinct and   $e_i=y_ry_{i}$ for each  $1\leq i\leq p$, where $r < i$; \\
\item For every edge $e_i$ with $2\leq i\leq p,$ there is a vertex $y_h$ with $0\leq h<i$ such that $\varphi(e_i)\in \bar\varphi(y_h).$
\end{itemize}
Clearly, a Tashkinov tree is indeed a tree of $G$. Tashkinov~\cite{Tashkinov-2000} proved that if $G$ is edge-$k$-critical  with $k \ge \D +1$, then $V(T)$ is elementary. In the above definition, if $e_i = y_0y_i$ for every $i$, i.e., $T$ is a star with $y_0$ as the center,  then $T$ is a Vizing fan. The classic result of Vizing~\cite{ss} show that for every Vizing fan $T$ the set $V(T)$ is elementary if $G$ is edge-$k$-critical for every $k \ge \D$, which  includes  edge-$\D$-critical graphs. In the definition of Tashkinov tree, if $e_i = y_{i-1}y_i$ for every $i$, i.e., $T$ is a path with end-vertices $y_0$ and $y_p$, then $T$ is a Kierstead path, which was introduced by Kierestead~\cite{Kierstead84}.  Kierstead proved that for every Kierstead path $P$ the set $V(P)$ is elementary if $G$ is an edge-$k$-critical with $k \ge \D +1$.  For simple graphs, following Kierstead's proof, Zhang~\cite{Zhang2000} noticed the following Lemma.
\begin{lem}\label{plong}{\em [Kierstead \cite{Kierstead84}, Zhang~\cite{Zhang2000}]}
Let $G$ be a graph with maximum degree $\Delta$ and ${\chi}'(G)= {\Delta}+1.$ Let $e_1\in E(G)$ be a critical edge. If $K=(y_0,e_1,y_1,\cdots,y_{p-1},e_p,y_p)$ is a Kierstead path with respect to $e_1$ and a coloring $\varphi\in\mathcal{C}^\Delta(G-e_1)$ such that $d(y_j)<\Delta$ for $j=2,\cdots,p$, then $V(K)$ is elementary with respect to $\varphi$.
\end{lem}

Kostochka and Stiebitz considered elementary property of Kierstead paths with four vertices and showed the following Lemma.

\begin{lem}\label{p4}{\em [Kostochka, Stiebitz~\cite{ss}]}
Let $G$ be a graph with maximum degree $\Delta$ and ${\chi}^\prime(G)= {\Delta}+1.$ Let $e_1\in E(G)$ be a critical edge and $\varphi\in\mathcal{C}^\Delta(G-e_1).$ If $K=(y_0,e_1,y_1,e_2,y_2,e_3,y_3)$ is a Kierstead path with respect to $e_1$ and $\varphi,$ then the following statements hold:
\begin{enumerate}
\item  $\bar{\varphi}(y_0)\cap \bar{\varphi}(y_1)=\emptyset;$
\item if $d(y_2)<\Delta,$ then $V(K)$ is elementary with respect to $\varphi;$
\item  if $d(y_1)<\Delta,$ then $V(K)$ is elementary with respect to $\varphi;$
\item  if $\Gamma=\bar{\varphi}(y_0)\cup \bar{\varphi}(y_1),$ then $|\bar{\varphi}(y_3)\cap \Gamma|\leq 1.$
\end{enumerate}
\end{lem}

In the definition of Tashkinov tree $T=(y_0, e_1, y_1, e_2, y_2, \cdots,  y_p)$.  We call $T$   a {\it broom} if $e_2=y_1y_2$ and for each $i \ge 3$, $e_i=y_2y_i$, i.e., $y_2$ is one of the end-vertices of $e_i$ for each $i \ge 3$. Moreover, we call a broom $T$ is a simple broom if $\phiv(e_i) \in \phibar(y_0)\cup\phibar(y_1)$ for each $i\ge 3$, i.e., $(y_0, e_1, y_1, e_2, y_2, e_i, y_i)$ is a Kierstead path. Chen, Chen and Zhao considered the elementary property  of simple brooms and gave the following Lemma.

\begin{lem}\label{broom}{\em [Chen, Chen, Zhao~\cite{cchen}]}
Let $G$ be an edge-$\D$-critical graph, $e_1=y_0y_1 \in E(G)$ and $\phiv \in\mathcal{C}^\Delta(G-e_1)$ and $B=\{y_0, e_1, y_1, e_2, y_2, \cdots, e_p,  y_p\}$ be a simple broom. If $|\bar{\varphi}(y_0)\cup \bar{\varphi}(y_1)|\geq 4$ and $\min\{d(y_1),d(y_2)\}< \Delta,$  then $V(B)$ is elementary with respect to $\varphi$.
\end{lem}

Brandt and Veldman gave the following result about the circumference of a graph.
\begin{lem}\label{lembrandt}\emph{[Brandt, Veldman \cite{bv}].}
Let $G\ne K_{1,n-1}$ be a graph of order $n$. If $d(x)+d(y)\ge n$ for any edge $xy$ of $G$, then the circumference of $G$ is $n-\max\{|S|-|N(S)|+1,0\}$, where $S$ is an independent set of $G$ with $S\cup N(S)\neq V(G)$.
\end{lem}

Using Lemma \ref{lembrandt}, Chen, Chen and Zhao showed the following Lemma.

\begin{lem}\label{lem12}
\emph{[Chen, Chen, Zhao \cite{cchen}].} Let $G$ be an edge-$\Delta$-critical graph of order $n$. If $d(x)+d(y)\geq n$ for any edge $xy$ of $G,$ then $G$ is Hamiltonian.
\end{lem}

The Bondy-Chv\'{a}tal closure $C(G)$ of a graph $G$ with order $n$ defined by Bondy and Chv\'{a}tal \cite{bm} is the maximal graph obtained from $G$ by consecutively adding the edges $xy$ if the degree sum of $x$ and $y$ is  at least $n$. They proved that $C(G)$ is well-defined and $C(G)$ is Hamiltonian if and only if $G$ is Hamiltonian. Brandt and Veldman gave the following result about the circumference of a graph $G$ and its closure $C(G).$

\begin{lem} \label{lem13} \emph{[Brandt, Veldman \cite{bv}].}
A graph $G$ has the same circumference as its closure $C(G).$
\end{lem}

\begin{lem}\label{lem14} \emph{[Chen, Ellingham, Saito, Song \cite{ce}]}.
Let $G$ be a bipartite graph with partite sets $X$ and $Y$. If for every $S\subseteq X$, $|N_G(S)|\ge \frac{3}{2}|S|$, then $G$ has a subgraph $H$ covering $X$  such that for every $x\in X$, $d_H(x)=2$ and for every $y\in Y$, $d_H(y)\le 2$.
\end{lem}

\begin{lem}\label{lem15}\emph{[Luo, Zhao \cite{lz}]}.
 An edge-$\D$-critical graph with at most 10 vertices is Hamiltonian.
\end{lem}


\section{ Proof of Theorem \ref{thml}}

By Lemma \ref{lem15}, we can assume that $n\ge 11$ and thus $\D\ge 20$. Suppose, on the contrary, there exists a non-Hamiltonian $\Delta$-critical graph $G$ of order $n$ with $\Delta\geq \frac{2}{3}n+12$.
If $d(x)+d(y)\geq n$ for any edge $xy\in E(G)$, then by Lemma \ref{lem12}, $G$ is Hamiltonian, a contradiction with the assumption. So
there exists an edge of $G$ with degree sum less than or equal to  $n-1$.
We will show  $\delta(G)\ge \frac{\D}{2}$ with a sequence of claims.

 \begin{cla} \label{clm1}
If $\delta(G)<\frac{\Delta}{2}$, then the followings hold:
\begin{itemize}
\item[{\bf a.}] $|V_{\ge \Delta-18}(G)|\ge \frac{107}{162}\Delta$;
\item [{\bf b.}] $|V_{\ge \Delta-18}(G)|\ge \frac{n}{2}$ provided $\D\le 100$;

\item [{\bf c.}] $|V_{\ge (1-\frac{13}{81})\Delta}(G)|\ge \frac{n}{2}$ provided $\D\ge 101$.
\end{itemize}
\end{cla}

\begin{proof} Let $x\in V(G)$ with $d(x)=\delta(G)$, and let $q$ be a positive number with $q\leq \Delta-10$.
By Lemma \ref{gen vizingfan}, there exists a vertex $y\in N(x)$ such that
\begin{equation}\label{thmeq1}
\sigma_q(x,y)> \Delta-\frac{d(x)}{2}-\frac{3d(x)-4}{\Delta-q}-\frac{4(d(x)-2)}{(\Delta-q)^2}.
\end{equation}

Plugging $d(x)\le \frac{\Delta-1}{2}$ and $q=\Delta-18$, we get \textbf{a} by the following inequalities.
\begin{equation*}
|V_{\ge \Delta-18}(G)|\ge \sigma_{\D-18}(x,y)
>\frac{3\Delta+1}{4}-\frac{3\Delta-11}{36}-\frac{2\Delta-10}{18^2} >\frac{107}{162}\Delta
\end{equation*}
When $\Delta\leq 100$, we have $\frac{107}{162}\Delta>\frac{3}{4}(\D-12)\ge \frac{n}{2}$ since $\Delta\geq \frac{2}{3}n+12$. Thus \textbf{b} holds.

Suppose $\Delta\geq 101$. Let $q=(1-\frac{13}{81})\Delta$. In this case $q<\Delta-10$.
By (\ref{thmeq1}), we have
\begin{eqnarray*}
|V_{\ge (1-\frac{13}{81})\Delta}(G)|&\ge& \sigma_{(1-\frac{13}{81})\Delta}(x,y)\\
&>&\frac{3\Delta+1}{4}-\frac{81(3\Delta-11)}{26\Delta}-81^2\frac{2\Delta-10}{(13\Delta)^2}\\
&\ge& \frac{n}{2}+9+\frac{1}{4}-\frac{243}{26}
+(\frac{81\times 11}{26\D}-\frac{81^2\times 2}{13^2\D})+\frac{81^2\times 10}{13^2\D^2}\\
&>&
\frac{n}{2}-0.0962-\frac{43.376}{\Delta}+\frac{388.22}{\Delta^2}>\frac{n}{2}-0.48761.
\end{eqnarray*}
Since $|V_{\ge (1-\frac{13}{81})\Delta}(G)|$ is an integer, we have $|V_{\ge (1-\frac{13}{81})\Delta}(G)|\geq\frac{n}{2}$.
So we have \textbf{c}.
\end{proof}

\begin{cla}\label{clm5}
If $\delta(G)\geq \frac{\Delta}{2}$,
then the followings hold:
\begin{itemize}
\item[{\bf a.}] $|V_{\ge \Delta-18}(G)|\ge \frac{107}{162}\Delta$;
\item [{\bf b.}] $|V_{\ge \Delta-18}(G)|\ge \frac{n}{2}$ provided $\D\le 100$;

\item [{\bf c.}] $|V_{\ge (1-\frac{13}{81})\Delta}(G)|\ge \frac{n}{2}$ provided $\D\ge 101$.
\end{itemize}
\end{cla}

\begin{proof}
Note that there exists  an edge of $G$ with degree sum less than or equal to $n-1$.  Let $xy\in E(G)$ such that $d(x)+d(y)\leq n-1<\frac{3}{2}\D$. For any $q\leq \Delta-10$, applying Lemma \ref{gen broom}, there exists a pair of adjacent vertices $\{u,v\}$, where $u\in \{x,y\}$ and $v\notin \{x,y\}$, such that
\begin{eqnarray}\label{thmeq2}
\sigma_q(u,v) &>&  \frac{3\Delta -d(x)-d(y)}{2}-\frac{3(d(x)+d(y)-\Delta)-2}{\Delta-q}-\frac{4(d(x)+d(y)-\Delta-2)}{(\Delta-q)^2}\nonumber\\
&>& \frac{3\Delta -n+1}{2}-\frac{3(n-\Delta)-5}{\Delta-q}-\frac{4(n-\Delta-3)}{(\Delta-q)^2}.
\end{eqnarray}

Plugging $q=\Delta-18$ and $n\leq \frac{3}{2}(\D-12)$ in (\ref{thmeq2}), we get \textbf{a} by the following inequalities.
$$
|V_{\ge \Delta-18}(G)|\ge \sigma_{\Delta-18}(u,v)> \frac{3\Delta+38
 }{4}-\frac{3\Delta-118}{36}-\frac{2\Delta-84}{18^2}>\frac{107}{162}\Delta.
$$
When $\Delta\leq 100$, we have $\frac{107}{162}\Delta>\frac{3}{4}(\D-12)\ge \frac{n}{2}$ since $\Delta\geq \frac{2}{3}n+12$. Thus \textbf{b} holds.

Suppose $\Delta\geq 101$. Let $q=(1-\frac{13}{81})\Delta$. In this case $q<\Delta-10$.
By (\ref{thmeq2}), we have
\begin{eqnarray*}
|V_{\ge (1-\frac{13}{81})\Delta}(G)|\ge \sigma_{(1-\frac{13}{81})\Delta}(x,y)
&>& \frac{3\Delta+38
 }{4}-81\frac{3\Delta-118}{26\Delta}-81^2\frac{2\Delta-84}{(13\Delta)^2}\\
&>& \frac{n}{2}+9+\frac{38}{4}-\frac{243}{26}+\frac{49005}{13^2\Delta}+\frac{81^2\times84}{(13\Delta)^2}> \frac{n}{2}.
\end{eqnarray*}
So we have \textbf{c}.
\end{proof}

Recall that $C(G)$ is the closure obtained from $G$ by consecutively adding edges between vertices whose degree sum greater than or equal to $n$. We have the following claim.

\begin{cla}\label{clm2}
In $C(G)$, $V_{\geq \frac{\Delta}{2}}(G)$ is a clique.
\end{cla}
\begin{proof}
Let $u\in V(G)$ with $d(u)\ge  \frac{\Delta}{2}$. For any vertex $v\in V_{\ge \Delta-18}(G)$, we have $d(u)+d(v)\geq \frac{3\Delta}{2}-18\ge n$, thus $uv\in E(C(G))$. It follows that $d_{C(G)}(u)\geq|V_{\geq \Delta-18}(G)|$.
By Claims \ref{clm1} and \ref{clm5}, we have $d_{C(G)}(u)\geq\frac{107}{162}\Delta$.
If $\Delta\leq 100$, then $d_{C(G)}(u)\geq \frac{n}{2}$, thus $V_{\geq \frac{\Delta}{2}}(G)$ is a clique of $C(G)$.
Now we assume $\Delta\ge 101$.
For any vertex $w\in V_{\geq (1-\frac{13}{81})\Delta}(G)$, we have
 $$d_{C(G)}(u)+d_{C(G)}(w)\geq d_{C(G)}(u)+d(w)\geq \frac{107}{162}\Delta+(1-\frac{13}{81})\Delta=\frac{3}{2}\Delta>n,$$
which implies that $uw\in E(C(G))$ for any vertex $w\in V_{\geq (1-\frac{13}{81})\Delta}(G)$. Thus $d_{C(G)}(u)\geq|V_{\ge (1-\frac{13}{81})\Delta}(G)|$.
By Claims \ref{clm1}.\textbf{c} and \ref{clm5}.\textbf{c}, we have $d_{C(G)}(u)\geq \frac{n}{2}$.
Hence, $V_{\geq \frac{\Delta}{2}}(G)$ is a clique of $C(G)$.
\end{proof}

\begin{cla}\label{clm3}
If $\delta(G)< \frac{\Delta}{2}$, then $|N(X)|\ge 2|X|$ for any $X\subseteq V_{<\frac{\Delta}{2}}(G)$.
\end{cla}

\begin{proof}
Let $X$ be a subset of $V_{<\frac{\Delta}{2}}(G)$. Since $G$ is edge-$\D$-critical, each edge of $G$ has degree sum greater than $\D+2$. Thus $X$ is an independent set of $G$, so $X\cap N(X)=\emptyset$. Let $H$ be the bipartite graph induced by the edges with one end-vertex in $X$ and the other in $N(X)$.
Recall that $p_{min}(x) :=  \min_{y\in N(x)}\sigma(x,y)-\Delta+d(x)-1$
and $p (x)  :=  \min \{p_{min}(x), \lfloor\frac{d(x)}{2}\rfloor-1\}$. For each $x\in X$, let $N_1(x)=\{y\in N(x):\sigma(x,y)\geq \Delta-p(x)-1\}$ and $N_2(x)=N(x)\setminus N_1(x)$.

For each $x\in X$ and $y\in N(x)$. By Lemma \ref{p}, $x$ has at least $d(x)-p(x)-1$ neighbors $y$ for which $\sigma(x,y)\geq \Delta-p(x)-1$.
Thus $|N_1(x)|\geq d(x)-p(x)-1$.
Since $2\Delta-d(x)-d(y)+2 >\frac{\Delta}{2}$, we have $\sigma(x,y)\le \sigma_{\frac{\Delta}{2}}(x,y)$.
Thus for each $y\in N_1(x)$ we have
$$d_H(y)\leq d(y)-\sigma(x,y)\leq d(y)-(\Delta-p(x)-1)\leq p(x)+1,$$
and for each $y\in N(x)$ we have
$$d_H(y)\leq d(y)-\sigma(x,y)\leq d(y)-(\Delta-d(x)+p(x)+1)\leq d(x)-p(x)-1.$$

For each edge $xy\in E(H)$ with $x\in X$ and $y\in N(X)$, we define $M(x,y)=\frac{1}{d_H(y)}$. Then we have
$$\sum_{xy\in E(H)}M(x,y)=\sum_{y\in N(X)}\sum_{x\in N(y)}\frac{1}{d_H(y)}=\sum_{y\in N(X)}1=|N(X)|.$$
On the other hand,
\begin{eqnarray*}\sum_{xy\in E(H)}M(x,y)&=&\sum_{x\in X}\sum_{y\in N(x)}\frac{1}{d_H(y)}=\sum_{x\in X}\left(\sum_{y\in N_1(x)}\frac{1}{d_H(y)}+\sum_{y\in N_2(x)}\frac{1}{d_H(y)}\right)\\
&\ge&\sum_{x\in X}\left(\frac{d(x)-p(x)-1}{p(x)+1}+\frac{p(x)+1}{d(x)-p(x)-1}\right)
\ge \sum_{x\in X}2=2|X|.\end{eqnarray*}
Therefore $|N(X)|\geq 2|X|$.
\end{proof}

\begin{cla}\label{clm4}
$\delta(G)\ge \frac{\Delta}{2}.$
\end{cla}

\begin{proof}
Suppose, on the contrary, $\delta(G)<\frac{\Delta}{2}$.
By Claim \ref{clm3}, we have $|N(X)|\geq 2|X|$ for any $X\subseteq V_{<\frac{\Delta}{2}}(G)$. Thus by Lemma \ref{lem14},
$G$ has a subgraph $H$ covering $V_{<\frac{\Delta}{2}}(G)$ such that for every $x\in V_{<\frac{\Delta}{2}}(G)$, $d_H(x)=2$ and for every $y\in V_{\ge\frac{\Delta}{2}}(G)$, $d_H(y)\le 2$.
That is,
there exist some vertex-disjoint paths $P_1,..., P_s$ covering $V_{<\frac{\Delta}{2}}(G)$ such that the end-vertices of $P_i$ belong to $V_{\geq\frac{\Delta}{2}}(G)$ for all $1\le i\le s$. Therefore, we can insert each vertex of $V_{<\frac{\Delta}{2}}$ into
the subgraph induced by $V_{\geq \frac{\Delta}{2}}(G)$.
By Claim \ref{clm2}, $V_{\geq \frac{\Delta}{2}}(G)$ is a clique of $C(G)$. Thus $C(G)$ is Hamiltonian. So by Lemma \ref{lem13}, $G$ is Hamiltonian, a contradiction.
\end{proof}

Now we complete the proof of our main theorem. By Claims \ref{clm2} and \ref{clm4}, $V(G)$ is a clique of $C(G)$.
So $C(G)$ is Hamiltonian, and by Lemma \ref{lem13}, $G$ is Hamiltonian, giving a contradiction.



\section{ Proofs of Lemmas \ref{gen vizingfan} and \ref{gen broom} }

\subsection{Proof of Lemma \ref{gen vizingfan}}

\noindent \textbf{Lemma \ref{gen vizingfan}.}
Let $xy$ be an edge in an edge-$\Delta$-critical graph $G$ and $q$  a positive number. If $d(x)<\frac{\Delta}{2}$ and $q\leq \Delta-10$, then
there exists a vertex $z\in N(x)\setminus\{y\}$ such that $\sigma_q(x,y)+\sigma_q(x,z)>2\Delta-d(x)-\frac{6d(x)-8}{\Delta-q}-\frac{8(d(x)-2)}{(\Delta-q)^2}.$
\begin{proof}

Let graph $G$, edge $xy\in E(G)$ and $q$ be defined in Lemma~\ref{gen vizingfan}.  A neighbor $z\in N(x)\setminus \{y\}$ is called {\it feasible} if there exits a coloring $\phiv\in \mathcal{C}^{\D}(G-xy)$ such that $\phiv(xz) \in \phibar(y)$, and such a coloring $\phiv$ is called a {\it $z$-feasible coloring}. Denote by
 $\mathcal{C}_z$ the set of all $z$-feasible colorings.
 For each  $\phiv\in \mathcal{C}_z$, let
\begin{eqnarray*}
Z(\phiv)  & = & \{v\in N(z)\setminus\{x\} \,:\,  \varphi(vz)\in \bar{\varphi}(x)\cup \bar{\varphi}(y)\}, \\
C_z(\phiv) & = & \{\varphi(vz) \,:\, v\in Z(\phiv) \mbox{ and } d(v) < q\}, \\
Y(\phiv) & =& \{v\in N(y)\setminus\{x\}\,:\,  \varphi(vy)\in \bar{\varphi}(x)\cup \bar{\varphi}(z)\}, \mbox{ and } \\
C_y(\phiv)  & = & \{\varphi(vy) \,:\,  v\in Y(\phiv) \mbox{ and } d(v) < q\}.
\end{eqnarray*}
Note that $Z(\phiv)$ and $Y(\phiv)$ are vertex sets while
$C_z(\phiv)$ and $C_y(\phiv)$ are color sets.
For each color $k\in \phiv(z)$, let $z_k$ be the unique vertex in $N(z)$ such that $\phiv(zz_k) = k$. Similarly, we define $y_k$ for each $k\in \phiv(y)$. Let
$T_0(\phiv) = \{k \in \phiv(x)\cap\phiv(y)\cap \phiv(z) \,:\,
d(y_k) < q \mbox{ and } d(z_k) < q \}.$

Since $G$ is edge-$\D$-critical, $\{x, y, z\}$ is elementary with respect to $\phiv$.  So $\phibar(x)$, $\phibar(y)$, $\phibar(z)$ and $\phiv(x)\cap \phiv(y)\cap \phiv(z)$ are mutually exclusive, and
$$
\phibar(x)\cup \phibar(y) \cup \phibar(z)\cup (\phiv(x)\cap\phiv(y)\cap \phiv(z)) = \{1, 2, \dots, \D\}.
$$
Recall that  $\sigma_q(x,y)$ and $\sigma_q(x,z)$ are number of vertices with degree $\ge q$ in $N(y)\setminus\{x\}$ and $N(z)\setminus \{x\}$, respectively. So, the following inequalities hold.
\begin{eqnarray*}
&&\sigma_q(x,y) + \sigma_q(x,z) \\
& \ge & |Y(\phiv)|-|C_y(\phiv)| + |Z(\phiv)| -|C_z(\phiv)| + |\phiv(x)\cap \phiv(y)\cap\phiv(z)| -|T_0(\phiv)|\\
& = & |\phibar(x)\cup\phibar(z)| + |\phibar(x)\cup\phibar(y)|-1 +
|\phiv(x)\cap\phiv(y)\cap \phiv(z)|\\
& & - |C_y(\phiv)| -|C_z(\phiv)| -|T_0(\phiv)|\\
& = & \D + |\phibar(x)| - |C_y(\phiv)| -|C_z(\phiv)| -|T_0(\phiv)|-1\\
& = &
2\D -d(x) -|C_y(\phiv)| -|C_z(\phiv)| -|T_0(\phiv)|
\end{eqnarray*}
So, Lemma~\ref{gen vizingfan} follows the two statements below.
\begin{itemize}
\item [{\bf I.}] For any $\phiv\in \mathcal{C}_z$, $|C_z(\phiv)|< \frac{d(x)-2}{\Delta-q}$ and $|C_y(\phiv)| < \frac{d(x)-2}{\Delta-q}$;
\item [{\bf II.}] there exists a $\phiv\in \mathcal{C}_z$ such that
$|T_0(\phiv)| \le \frac{4d(x)-4}{\Delta-q}+\frac{8(d(x)-2)}{(\Delta-q)^2}$.
\end{itemize}

For every  $\phiv\in \mathcal{C}_z$, let $\phiv^d\in \mathcal{C}^{\D}(G-xz)$ obtained from $\phiv$ by assigning $\phiv^d(xy) = \phiv(xz)$ and keeping all colors on other edges unchange.  Clearly, $\phiv^d$ is a $y$-feasible coloring and $Z(\phiv^d)=Z(\phiv)$, $Y(\phiv^d) = Y(\phiv)$, $C_z(\phiv^d) =C_z(\phiv)$ and $C_y(\phiv^d) = C_y(\phiv)$.  We call $\phiv^d$ the {\it dual coloring} of $\phiv$. Considering dual colorings, we see that some properties for vertex $z$ is also held for vertex $y$.

Let $z\in N(x)\setminus\{y\}$ be a feasible vertex and $\phiv\in \mathcal{C}_z$. By the definition of $Z(\phiv)$,  $G[\{x, y, z\}\cup Z(\phiv)]$ contains a simple broom. Since $d(x)<\frac{\Delta}{2}$ and $\frac{\Delta}{2}\ge 1$, we have $|\bar{\varphi}(x)\cup \bar{\varphi}(y)|\geq 4$.
Thus by Lemma \ref{broom} the set $\{x, y, z\}\cup Z(\phiv)$ is elementary with respect to $\phiv$. Counting the number of missing colors of vertices in the set $\{x,y,z\}\cup Z(\phiv)$, we obtain $(\Delta-q)|C_z(\phiv)|+|\bar{\varphi}(x)|+|\bar{\varphi}(y)|+|\bar{\varphi}(z)|<\sum_{v\in \{x,y,z\}\cup Z(\phiv)}|\phibar(v)|\leq \Delta$, which implies that $|C_z(\phiv)|<\frac{d(x)-2}{\Delta-q}$.
By considering its dual $\phiv^d$, we have $|C_y(\phiv)| =|C_y(\phiv^d)|<\frac{d(x)-2}{\Delta-q}$. Hence, {\bf I} holds.

The proof of {\bf II} is much more complicated and will be placed in a separated section.
A  coloring $\phiv\in \mathcal{C}_z$ is called {\it optimal} if over all feasible colorings the followings hold:

1. $|C_z(\phiv)|+|C_y(\phiv)|$ is maximum;

2. subject to 1, $|C_z(\phiv)\cap C_y(\phiv)|$ is minimum.

\subsubsection{Proof of {\bf II}.}

Suppose, on the contrary, $|T_0(\phiv)| > \frac{4d(x)-4}{\Delta-q}+\frac{8(d(x)-2)}{(\Delta-q)^2}$ for every $\phiv\in \mathcal{C}_z$.
Let $T_0(\phiv)=\{k_1,\cdots,k_{|T_0(\phiv)|}\}$
and $V_{T_0(\phiv)}=\{z_{k_i}\in N(z)\,:\,k_i\in T_0(\phiv)\}\cup\{y_{k_i}\in N(y)\,:\,k_i\in T_0(\phiv)\}$.
Let $\phiv$ be an optimal feasible coloring and assume, without loss of generality, $\phiv(xz) =1$.
Let $Z=Z(\phiv)$, $Y=Y(\phiv)$, $C_z = C_z(\phiv)$,  $C_y = C_y(\phiv)$ and $T_0 = T_0(\phiv)$ if the coloring $\phiv$ is clearly referred. Let $R=C_z\cup C_y$.

{\flushleft \bf Claim A.} For each $i\in \phibar(x)\setminus R$ and
$k\in T_0$,  we have $P_x(i,k,\varphi )$ contains both $y$ and $z$.

\begin{proof} We first show that $z\in V(P_x(i,k,\varphi ))$. Otherwise, $P_z(i, k, \phiv)$ is disjoint with $P_x(i, k, \phiv)$. Let $\varphi^\prime=\varphi / P_z(i,k,\varphi )$.  Since $i, k\ne 1$, $\phiv^\prime$ is also feasible. Since colors in $R$ are unchanged and $d(z_k)< q$, we have $C_z(\phiv^\prime) = C_z\cup \{i\}$ and $C_y(\phiv^\prime) \supseteq C_y$, giving a contradiction to
the maximality of $|C_y| + |C_z|$.  By considering the dual $\phiv^d$, we can verify that $y\in V(P_x(i,k,\varphi ))$.
\end{proof}

{\flushleft \bf Claim B.} If there exist three vertices $u_1,u_2,u_3\in V_{T_0}$ and two colors $\alpha\in \varphi(x)\cup R,\beta\in \bar{\varphi} (x)\setminus R$ such that $\alpha\in \bar{\varphi}(u_1)\cap \bar{\varphi}(u_2)\cap \bar{\varphi}(u_3)$ and $\beta\in \varphi(u_1)\cap \varphi(u_2)\cap \varphi(u_3)$, then there exists a vertex $u\in \{u_1,u_2,u_3\}$ and an optimal feasible coloring $\varphi^\prime=\phiv/P_u(\alpha,\beta,\phiv)$ such that $\beta\in \bar{\varphi^\prime}(u)$.

\begin{proof}
We first consider the case of $\alpha\in \phiv(x)\setminus R$.
Since $\alpha\in \bar{\varphi}(u_1)\cap \bar{\varphi}(u_2)\cap \bar{\varphi}(u_3)$, we may assume $P_x(\alpha,\beta,\phiv)$ is disjoint with $P_{u_1}(\alpha,\beta,\phiv)$.
Let $\varphi_1=\varphi/ P_{u_1}(\alpha,\beta,\varphi)$. We claim that $\phiv_1$ is still feasible. It is easy to see that $\phiv_1$ is feasible when $\alpha\ne 1$.
Now we suppose $\alpha=1$. In this case, $P_x(\alpha,\beta,\varphi)=P_y(\alpha,\beta,\varphi)$.
Then $z\in P_x(\alpha,\beta,\varphi)$ as $\phiv(xz)=1$. Thus
$\phiv_1$ is feasible.
Since $\alpha, \beta\notin R$,
we have $C_y(\phiv_1) = C_y$ and $C_z(\phiv_1) = C_z$. So, $\phiv_1$ is also optimal and $\beta\in \bar{\varphi_1 }(u_1)$.

We now suppose $\alpha\in R$. By the definition of $R$, there exists a vertex $v\in \{x,y,z\}$ such that $\alpha\in \phibar(v)$. Clearly, $P_v(\alpha,\beta,\phiv)=P_x(\alpha,\beta,\phiv)$ when $v=x$.
We claim that $P_v(\alpha,\beta,\phiv)=P_x(\alpha,\beta,\phiv)$ when $v\in \{y,z\}$.
Otherwise, if $v=y$ then $\phiv/P_x(\alpha,\beta, \phiv)$ would lead a $\D$-coloring of $G$ and if $v=z$ then $\phiv^d/P_x(\alpha,\beta, \phiv)$ would lead a $\D$-coloring of $G$, giving a contradiction in either case.
Thus $V(P_x(\alpha,\beta,\phiv))\cap \{u_1,u_2,u_3\}=\emptyset$.
Let $w$ be a vertex of $\{y,z\}$ such that $\alpha\in C_w$. We may assume $P_{u_1}(\alpha,\beta,\phiv)$ is disjoint with $P_w(\alpha,\beta,\phiv)$.
If $\alpha\notin C_z\cap C_y$, then $\varphi_1=\varphi/ P_{u_1}(\alpha,\beta,\varphi)$ is the required coloring such that $\beta\in \bar{\varphi_1 }(u_1)$.
Now we suppose that $\alpha\in C_z\cap C_y$. Thus $\alpha\in \phibar(x)$.
We claim that $P_{z}(\alpha,\beta,\varphi)=P_{y}(\alpha,\beta,\varphi)$. For otherwise, let $\varphi_1=\varphi/ P_{z}(\alpha,\beta,\varphi)$, we have $\phiv_1(x)=\phiv(x)$, $C_z(\phiv_1)=(C_z\cup \{\beta\})\setminus \{\alpha\}$ and $C_y(\phiv_1)=C_y$, which implies that
$|C_z(\phiv_1)|+|C_y(\phiv_1)|=|C_z|+|C_y|$ and $|C_z(\phiv_1)\cap C_y(\phiv_1)|=|C_z\cap C_y|-1$, which leads a contradiction to the minimality of $|C_z\cap C_y|$.
Thus
$\varphi_1=\varphi/ P_{u_1}(\alpha,\beta,\varphi)$ is feasible, $C_y(\phiv_1) = C_y$ and $C_z(\phiv_1) = C_z$.
So $\varphi_1$ is also optimal and $\beta\in \bar{\varphi_1 }(u_1)$.
\end{proof}

Recall that we have assumed $|T_0| \ge \frac{4d(x)-4}{\Delta-q}+\frac{8(d(x)-2)}{(\Delta-q)^2}$. Let $T=\{k_1, \cdots, k_{|T|}\}$ be a subset of $T_0$ with $|T|=\lfloor\frac{4d(x)-4}{\Delta-q}+\frac{8(d(x)-2)}{(\Delta-q)^2}\rfloor+1$.
Let
 \begin{eqnarray*}
V_T & = &\{z_{k_1},z_{k_2},\cdots, z_{k_{|T|}}\}\cup\{y_{k_1},y_{k_2},\cdots, y_{k_{|T|}}\},\\
W(\phiv) & = & \{u\in V_T\,:\, \phibar(u) \subseteq\phiv(x)\cup  R \},\\
M(\phiv) & = & V_T-W(\phiv)=\{u\in V_T \,:\, \phibar(u)\cap(\phibar(x)\setminus R)\neq \emptyset\},\\
E_T & = & \{zz_{k_1}, zz_{k_2},\cdots, zz_{k_{|T|}}, yy_{k_1},
yy_{k_2},\cdots, yy_{k_{|T|}}\},\\
E_W(\phiv) & = & \{e\in E_T \,:\,  \mbox{ $e$ is incident to a vertex in $W(\phiv)$}\},\mbox{ and } \\
E_M(\phiv) & = & E_T-E_W(\phiv)=\{e\in E_T\,:\, \mbox{ $e$  is incident to a vertex
in $M(\phiv)$} \}.
\end{eqnarray*}
Let $W=W(\phiv),M=M(\phiv),E_W=E_W(\phiv)$ and $E_M=E_M(\phiv)$ if the coloring $\phiv$ is clearly referred.

We assume that $|E_M|$ is maximum over all optimal feasible colorings $\phiv$ and all subsets of $T_0$ with order $|T|$. For each $v\in M$, pick a color $\alpha_v\in \phibar(v)\cap (\phibar(x)\setminus R)$. Let $C_M=\{\alpha_v \,:\, v\in M\}$. Clearly, $|C_M| \le |M|$.  Note that $\{z_{k_1},z_{k_2},\cdots, z_{k_{|T|}}\}\cap\{y_{k_1},y_{k_2},\cdots, y_{k_{|T|}}\}$ may be not empty, we have $\frac{|E_W|}{2}\leq |W|\leq |E_W|$ and $\frac{|E_M|}{2}\leq |M|\leq |E_M|$.

Since
$|T|=\lfloor\frac{4d(x)-4}{\Delta-q}+\frac{8(d(x)-2)}{(\Delta-q)^2}\rfloor+1$,
by \textbf{I}, we have
\begin{equation}\label{ineq1}
(\D-q)\left\lceil\frac{|T|}{2}\right\rceil\geq 2d(x)-2+\frac{4(d(x)-2)}{\Delta-q}>2|\phiv(x)\cup R|.
\end{equation}
Since $d(x)<\frac{\Delta}{2}$ and $q\leq \Delta-10$, by \textbf{I}, we have  $|\bar{\varphi }(x)\setminus R|>\Delta-d(x)+1-2(\frac{d(x)-2}{\Delta-q})>\frac{2\Delta+7}{5}$.
and $|T|+\frac{2|\varphi(x)\cup R|}{\Delta-q-1}<\frac{\Delta-2}{5}+\frac{\Delta-4}{25}+\frac{6\Delta-14}{45}+1=\frac{84\Delta+29}{225}$.
Since $\frac{2\Delta+7}{5}>\frac{84\Delta+29}{225}$, the following inequality holds.
\begin{equation}\label{ineq2}
|\bar{\varphi }(x)\setminus R|>|T|+\frac{2|\varphi(x)\cup R|}{\Delta-q-1}
\end{equation}

{\flushleft \bf Claim C.} Suppose there is no color $k\in T$ such that there are three distinct colors $i,j,l\in \phibar(x)\setminus R$ with $i\in \phibar(z_k)$ and $j,l\in \phibar(y_k)$. Then there exists an optimal feasible coloring such that $|E_M|\geq |T|+\frac{2|\varphi(x)\cup R|}{\Delta-q-1}.$

\begin{proof}
Suppose, on the contrary, $|E_M|< |T|+\frac{2|\varphi(x)\cup R|}{\Delta-q-1}.$
Since $|C_M|\le |E_M|$, by (\ref{ineq2}), there exists a color $\beta\in  \bar{\varphi }(x)\setminus (R\cup C_M)$. Clearly, $\beta\in \phiv(u)$ for each $u\in W$.

First we claim that $|E_M|\geq |T|+1$. If not, then $|E_M|\leq |T|$ and $|E_W|\geq |T|$ since $|E_M|+|E_W|=2|T|$. Thus by (\ref{ineq1}) we have $\sum_{v\in W}|\bar{\varphi}(v)|> (\Delta-q)\left\lceil\frac{|T|}{2}\right\rceil>
2|\varphi(x)\cup R|$. By the definition of $W$, $\bar{\varphi}(v)\subseteq \varphi(x)\cup R$ for every $v\in W$, by the Pigeonhole Principle,  there exists three vertices $u_1,u_2,u_3\in W$ and a color $\alpha\in \phiv(x)\cup R$ such that $\alpha\in \bar{\varphi}(u_1)\cap \bar{\varphi}(u_2)\cap \bar{\varphi}(u_3)$.
We note that $\beta\in \varphi(u_1)\cap \varphi(u_2)\cap \varphi(u_3)$ since $u_1,u_2,u_3\in W$. Applying Claim B with $\alpha$ and $\beta$, there exists a vertex $u\in \{u_1,u_2,u_3\}$ and an optimal feasible coloring $\varphi^\prime=\varphi/ P_{u}(\alpha,\beta,\varphi)$ such that $\beta\in \bar{\varphi^\prime}(u)$. Since $u\in W$ and $\beta\notin C_M$, we have $|E_M(\varphi^\prime)|>|E_M|$, a contradiction to the maximality of $|E_M|$.

We may assume that $|E_M|=|T|+p$, thus $|E_W|=|T|-p$. Since $|E_M|=|T|+p$, there exists a $p$-element set $Y_1=\{y_{k}\, :\,k\in T \mbox{\ and } z_{k},y_{k}\in M \}$.  For each $y_{k}\in Y_1$, by Claim A, we have $(\phibar(x)\setminus R)\cap \bar{\varphi }(y_{k})\cap \bar{\varphi }(z_{k})=\emptyset$. Since there do not exist color $k\in T$ such that there exist $i,j,l\in \phibar(x)\setminus R$ with $i\in  \bar{\varphi }(z_k)$, $j,l\in  \bar{\varphi }(y_k)$, we have  $|(\phibar(x)\setminus R)\cap \bar{\varphi }(y_{k})|=1$ for each $y_{k}\in Y_1$. Thus by (\ref{ineq1}) we have
$\sum_{v\in W\cup Y_1}|\bar{\varphi}(v)\cap (\phiv(x)\cup R)|> (\Delta-q)\left\lceil\frac{|T|-p}{2}\right\rceil+(\Delta-q-1)p\geq (\Delta-q)\left\lceil\frac{|T|}{2}\right\rceil+(\frac{\Delta-q}{2}-1)p>
2|\varphi(x)\cup R|$.
Hence there exist three vertices $u_1,u_2,u_3\in W\cup Y_1$ and a color $\alpha\in \varphi(x)\cup R$ such that $\alpha\in \bar{\varphi}(u_1)\cap \bar{\varphi}(u_2)\cap \bar{\varphi}(u_3)$. Since $\beta\notin  C_M$ and $|(\phibar(x)\setminus R)\cap \bar{\varphi }(y_{k})|=1$ for each $y_{k}\in Y_1$, we have $\beta\in \varphi(u_1)\cap \varphi(u_2)\cap \varphi(u_3)$. Applying Claim B, there exists a vertex $u\in \{u_1,u_2,u_3\}$ and an optimal feasible coloring $\varphi^\prime=\varphi/ P_{u}(\alpha,\beta,\varphi)$ such that $\beta\in \bar{\varphi^\prime}(u)$. Thus $u\in M(\varphi^\prime)$. If $u\in W$, then $|E_M(\varphi^\prime)|>|E_M|$, a contradiction. Thus $u\in Y_1$. We may assume that $u=y_{k_a}$. So there exist three colors $i,j,\beta\in \bar{\varphi^\prime }(x)\setminus R$ such that $i\in  \bar{\varphi^\prime}(z_{k_a})$, $j,\beta\in  \bar{\varphi^\prime}(y_{k_a})$, giving a contradiction.
Thus Claim C holds.
\end{proof}

{\flushleft \bf Claim D.} There exists a color $k\in T$ and three distinct colors $i,j,l\in \phibar(x)\setminus R$ such that $i\in  \bar{\varphi }(z_k)$, $j,l\in  \bar{\varphi }(y_k)$.

\begin{proof}
We first note that if there exist $i,j\in \phibar(x)\setminus R$ such that $i\in  \bar{\varphi }(z_k)$ and $j\in  \bar{\varphi }(y_k)$, then $i\neq j$; for otherwise, by Claim A, the path $P_{x}(i,k,\varphi )$ contains three endvertices $x,z_k$ and $y_k$, a contradiction.

Suppose that Claim D does not hold.
By Claim C, we have $|E_M|\geq |T|+\frac{2|\varphi(x)\cup R|}{\Delta-q-1}$.
So there exists a $\lceil\frac{2|\varphi(x)\cup R|}{\Delta-q-1}\rceil$-element set  $Y_1=\{y_{k}\,:\,k\in T \mbox{\ and } z_{k},y_{k}\in M \}$ such that $|\bar{\varphi }(y_{k})\cap(\phibar(x)\setminus R)|=1$ for each $y_{k}\in Y$.
Then we have $\sum_{v\in Y_1}|\bar{\varphi}(v)\cap (\phiv(x)\cup R)|> (\Delta-q-1)|Y_1|= 2|\varphi(x)\cup R|$.
Thus there exist three vertices $u_1,u_2,u_3\in Y_1$ and a color $\alpha\in \varphi (x)\cup R$ such that $\alpha\in \bar{\varphi}(u_1)\cap \bar{\varphi}(u_2)\cap \bar{\varphi}(u_3)$.
For each $y_{k}\in Y_1$, pick a color $\alpha_{k1}\in \phibar(y_{k})\cap (\phibar(x)\setminus R)$ and $\alpha_{k2}\in \phibar(z_{k})\cap (\phibar(x)\setminus R)$. Let $C_M^\prime=\{\alpha_{kj} \,:\, y_{k}\in Y_1 \mbox{\ and } 1\le j\le 2\}$. Clearly, $|C_M^\prime|=2|Y_1|$.
 By (\ref{ineq2}), we have $|\bar{\varphi}(x)\setminus R|>|C_M^\prime|$, thus there exists a color
$\beta\in \phibar(x)\setminus (R\cup C_M^\prime)$ such that $\beta\in \varphi(y_{k})$ for all $y_{k}\in Y_1$.
So $\beta\in \varphi(u_1)\cap \varphi(u_2)\cap \varphi(u_3)$.

Applying Claim B with colors $\alpha$ and $\beta$, there exists a vertex $u\in \{u_1,u_2,u_3\}$ and an optimal feasible coloring $\varphi^\prime=\varphi/ P_{u}(\alpha,\beta,\varphi)$ such that $\beta\in \bar{\varphi^\prime}(u)$. Then $u\in M(\phiv^\prime)$.
Since $\beta\notin C_M^\prime$, there exists a color $k_a\in T$ with $u=y_{k_a}$ and colors $i,j,\beta\in \bar{\varphi^\prime }(x)\setminus R$ such that $i\in  \bar{\varphi^\prime}(z_{k_a})$, $j,\beta\in  \bar{\varphi^\prime}(y_{k_a})$, giving a contradiction.
\end{proof}

Let $k,i,j,l$ be as stated in Claim D. Note that $P_x(l,1,\varphi)$ and $P_y(l,1,\varphi)$ are same and are disjoint from $P_{y_k}(l,1,\varphi)$.
 If $l\ne 1$, we consider coloring  $\phiv/P_{y_k}(l,1,\varphi)$, and rename it as $\phiv$. So we may assume $1\in \bar{\varphi }(y_k)$.

By Claim A, the paths $P_x(i,k,\varphi)$ and $P_x(j,k,\varphi)$ both contain $y,z$.  Since
$\phiv(yy_k)=\phiv(zz_k)=k$,  these two paths   also contain $y_k,z_k$. Since $i\in \bar{\varphi }(z_k)$, we have $x$ and $z_k$ are the two endvertices of $P_x(i, k, \phiv)$. So, $i\in \phiv(y)\cap\phiv(z)\cap\phiv(y_k)$.  Similarly, we have $j\in \phiv(y)\cap \phiv(z)\cap \phiv(z_k)$.
We now consider the following sequence of colorings of $G-xy$.

 Let $\phiv_1$ be obtained from $\phiv$ by assigning $\phiv_1(yy_k) = 1$.  Since $1$ is missing at both $y$ and $y_k$,  $\phiv_1$ is an edge-$\D$-coloring of $G-xy$. Now $k$ is missing at $y$ and $y_k$, $i$ is still missing at $z_k$.  Since $G$ is not $\D$-colorable, $P_x(i,k,  \phiv_1) = P_y( i, k, \phiv_1)$; otherwise $\phiv_1/P_y( i, k, \phiv_1)$ can be extended to an edge-$\D$-coloring of $G$ giving a contradiction.  Furthermore, $z_k, y_k\notin V(P_x(i, k, \phiv_1))$ since either $i$ or $k$ is missing at these two vertices, which in turn shows that $z\notin V(P_x(i, k, \phiv_1))$ since $\phiv_1(zz_k) =k$.

Let $\varphi _2=\varphi _1/P_x(i,k,\varphi_1 )$. We have $k\in \bar{\varphi _2}(x), i \in \bar{\varphi _2}(y)\cap\bar{\varphi _2}(z_k)$ and $ j\in \bar{\varphi _2}(x)\cap\bar{\varphi _2}(y_k)$. Since $G$ is not edge-$\D$-colorable, $P_x(i, j, \phiv_2) = P_y(i,j, \phiv_2)$ which  contains neither  $y_k$ nor $z_k$.

Let $\phiv_3 = \phiv_2/P_x(i, j, \phiv_2)$.  Then $k\in \bar{\varphi _3}(x)$ and $j\in \bar{\varphi _3}(y)\cap \bar{\varphi _3}(y_k)$.

Let $\varphi _4$ be obtained from $\varphi _3$ by recoloring $yy_k$ by $j$.  Then $1\in \bar{\varphi _4}(y)$, $\varphi _4(xz)=1$, $k\in \bar{\varphi _4}(x)$, $\varphi _4(zz_k)=k$. Since $\phiv_4(xz) = 1\in \phibar_4(y)$,  $\phiv_4$ is $z$-feasible.
Since $i,j,k\notin R=C_y\cup C_z$, the colors in $R$ are unchanged during this sequence of re-colorings, so $C_y(\phiv_4) \supseteq C_y$ and $C_z(\phiv_4) \supseteq C_z$.  Since $\phiv_4(zz_k)=k\in \phibar_4(x)$ and $d(z_k) < q$, we have $k \in C_z(\phiv_4)$. So, $C_z(\phiv_4) \supseteq C_z\cup \{k\}$. Therefore, $|C_y(\phiv_4)| + |C_z(\phiv_4)| \ge |C_y| + |C_z| +1$, giving a contradiction.
So \textbf{II} holds.
\end{proof}

\subsection{ Proof of Lemma \ref{gen broom} }
The proof of Lemma \ref{gen broom} follows similar structure of the proof of Lemma \ref{gen vizingfan}. Nevertheless the difference is big enough. For this reason, we give it.

\noindent \textbf{Lemma \ref{gen broom}.}
Let $x_1x_2$ be an edge in an  edge-$\Delta$-critical graph $G$ and $q$  a positive number. If $d(x_1)+d(x_2)<\frac{3}{2}\D$, $q\leq \Delta-10$ and $\delta(G)\geq \frac{\Delta}{2}$, then
there exists a pair of vertices $\{z,y\}$ with $z\in N(x_1)\setminus \{x_2\}$ and $y\in N(x_2)\setminus \{x_1,z\}$ such that
$\sigma_q(x_1,z)+\sigma_q(x_2,y)>
3\Delta -d(x_1)-d(x_2)-\frac{6(d(x_1)+d(x_2)-\Delta)-4}{\Delta-q}-\frac{8(d(x_1)+d(x_2)-\Delta-2)}{(\Delta-q)^2}
$.

\begin{proof}
Let graph $G$, edge $x_1x_2\in E(G)$, $q$ and $\delta(G)$ be defined in Lemma~\ref{gen broom}.  A pair of vertices $\{z,y\}$ with
$z\in N(x_1)\setminus \{x_2\}$ and $y\in N(x_2)\setminus \{x_1, z\}$ is called {\it feasible} if there exits a coloring $\phiv\in \mathcal{C}^{\D}(G-x_1x_2)$ such that $\varphi(x_1z)\in \bar{\varphi}(x_2)$ and $\varphi( x_2y)\in \bar{\varphi}(x_1)$, and such a coloring $\phiv$ is called a {\it $zy$-feasible coloring}. Denote by
 $\mathcal{C}_{zy}$ the set of all $zy$-feasible colorings.
 For each  $\phiv\in \mathcal{C}_{zy}$, let
\begin{eqnarray*}
Z(\phiv)  & = & \{v\in N(z)\setminus\{x_1\} \,:\,  \varphi(vz)\in \bar{\varphi}(x_1)\cup \bar{\varphi}(x_2)\cup \bar{\varphi}(y)\}, \\
C_z(\phiv) & = & \{\varphi(vz) \,:\, v\in Z(\phiv) \mbox{ and } d(v) < q\}, \\
Y(\phiv) & =& \{v\in N(y)\setminus\{x_2\}\,:\,  \varphi(vy)\in \bar{\varphi}(x_1)\cup \bar{\varphi}(x_2)\cup \bar{\varphi}(z)\}, \mbox{ and } \\
C_y(\phiv)  & = & \{\varphi(vy) \,:\,  v\in Y(\phiv) \mbox{ and } d(v) < q\}.
\end{eqnarray*}

Note that $Z(\phiv)$ and $Y(\phiv)$ are vertex sets while
$C_z(\phiv)$ and $C_y(\phiv)$ are color sets.
For each color $k\in \phiv(z)$, let $z_k \in N(z)$ such that $\phiv(zz_k) = k$. Similarly, we define $y_k$ for each $k\in \phiv(y)$. Let
$T_0(\phiv) = \{k \in \phiv(x_1)\cap \phiv(x_2)\cap\phiv(y)\cap \phiv(z) \,:\,
d(y_k) < q \mbox{ and } d(z_k) < q \}.$

Since $d(x_1)+d(x_2)<\frac{3}{2}\D$ and $\delta(G)\geq \frac{\Delta}{2}$, we have $d(x_1)<\Delta$ and $d(x_2)<\Delta$.
 We  assume that $\varphi(x_1z)=1$
 and $\varphi( x_2y)=2$.

First, we claim that $\phibar(x_1)$, $\phibar(x_2)$, $\phibar(y)$, $\phibar(z)$ and $\phiv(x_1)\cap\phiv(x_2)\cap \phiv(y)\cap \phiv(z)$ are mutually exclusive.
Let $\varphi_1$ be obtained from $\varphi$ by uncoloring $x_1z$ and coloring $x_1x_2$ with the color $1$, then $\{z,x_1,x_2,y\}$ forms a Kierstead path with respect to $\phiv_1$. By Lemma \ref{p4}, $\{z,x_1,x_2,y\}$ is elementary with respect to $\varphi_1$ as $d(x_1)<\D$. It follows that $\phibar_1(y)\cap \phibar_1(z)=\emptyset$, that is, $\phibar(y)\cap (\phibar(z)\cup \{1\})=\emptyset$. Hence, $\phibar(y)\cap \phibar(z)=\emptyset$.
Since $G$ is edge-$\D$-critical, $\{x_1, x_2, z\}$ and $\{x_1, x_2, y\}$ are both elementary with respect to $\phiv$.
The claim holds and
$$
\phibar(x_1)\cup \phibar(x_2)\cup \phibar(y) \cup \phibar(z)\cup (\phiv(x_1)\cap \phiv(x_2)\cap\phiv(y)\cap \phiv(z)) = \{1, 2, \dots, \D\}.
$$

Recall that  $\sigma_q(x_2,y)$ and $\sigma_q(x_1,z)$ are number of vertices with degree $\ge q$ in $N(y)\setminus\{x_2\}$ and $N(z)\setminus \{x_1\}$, respectively. So, the following inequalities hold.
\begin{eqnarray*}
&&\sigma_q(x_2,y) + \sigma_q(x_1,z) \\
& \ge & |Y(\phiv)|-|C_y(\phiv)| + |Z(\phiv)| -|C_z(\phiv)| + |\phiv(x_1)\cap \phiv(x_2)\cap \phiv(y)\cap\phiv(z)| -|T_0(\phiv)|\\
& = & |\phibar(x_1)\cup \phibar(x_2)\cup\phibar(z)| -1+ |\phibar(x_1)\cup \phibar(x_2)\cup\phibar(y)|-1 \\
& &+
|\phiv(x_1)\cap \phiv(x_2)\cap\phiv(y)\cap \phiv(z)| - |C_y(\phiv)| -|C_z(\phiv)| -|T_0(\phiv)|\\
& = & \D + |\phibar(x_1)| + |\phibar(x_2)|- |C_y(\phiv)| -|C_z(\phiv)| -|T_0(\phiv)|-2\\
& = &
3\D -d(x_1)-d(x_2) -|C_y(\phiv)| -|C_z(\phiv)| -|T_0(\phiv)|
\end{eqnarray*}
So, Lemma~\ref{gen broom} follows the two statements below.
\begin{itemize}
\item [{\bf I.}] For any $\phiv\in \mathcal{C}_{zy}$, $|C_z(\phiv)|< \frac{d(x_1)+d(x_2)-\Delta-2}{\Delta-q}$ and $|C_y(\phiv)| < \frac{d(x_1)+d(x_2)-\Delta-2}{\Delta-q}$;
\item [{\bf II.}] there exists a $\phiv\in \mathcal{C}_{zy}$ such that
$|T_0(\phiv)| \le \frac{4(d(x_1)+d(x_2)-\Delta)}{\Delta-q}+\frac{8(d(x_1)+d(x_2)-\Delta-2)}{(\Delta-q)^2}$.
\end{itemize}

In the remainder of the proof, we let $Z=Z(\phiv)$, $Y=Y(\phiv)$, $C_z = C_z(\phiv)$,  $C_y = C_y(\phiv)$ and $T_0 = T_0(\phiv)$ if the coloring $\phiv$ is clearly referred. Let $R=C_z\cup C_y$ and $R^\prime=C_z\cup C_y\cup \{1,2\}$. A  coloring $\phiv\in \mathcal{C}_{zy}$ is called {\it optimal} if over all feasible colorings the followings hold:

1. $|C_z|+|C_y|$ is maximum;

2. subject to 1, $|C_z\cap C_y|$ is minimum.

\subsubsection{Proof of {\bf I}.}

{\flushleft \bf Claim A1.}
For every $\varphi  \in \mathcal{C}_{zy}$,
the sets $\{x_1,x_2,z\}\cup Z$ and $\{x_1,x_2,y\}\cup Y$ are both elementary with respect to $\varphi$.

\begin{proof}
We first show that
\begin{equation}\label{eq3}
\mbox{\ for any $v\in Z$, $\{x_1,x_2,z,v\}$ is elementary with respect to $\varphi$.\ }
\end{equation}
This is clearly true if $\varphi(vz)\in \phibar(x_1)\cup\phibar(x_2)$. We now assume $\varphi(vz)\in \bar{\varphi}(y)$.
Since $\bar{\varphi}(y)\supseteq \{1, \varphi(vz)\}$, we have $d(y)<\Delta$.
Suppose $\{x_1,x_2,z,v\}$ is not elementary, let $\eta\in \bar{\varphi}(v)\cap (\bar{\varphi}(x_1)\cup\bar{\varphi}(x_2)\cup \bar{\varphi}(z))$. We let $\varphi(vz)=k$.
Note that $\bar{\varphi}(x_1)\cap \bar{\varphi}(y)=\emptyset$, we have $k\in \varphi(x_1)$.

First we consider the case $\eta\in\bar{\varphi}(x_1)$. Since $\bar{\varphi}(x_1)\cap \bar{\varphi}(y)=\emptyset$, we have $\eta \in \varphi(y)$. Thus $P_{x_1}(\eta,k,\varphi)=P_{y}(\eta,k,\varphi)$, for otherwise, let $\phiv_1=\phiv/P_{x_1}(\eta,k,\varphi)$, $\{x_1,x_2,y\}$ is not elementary with respect to $\phiv_1$ since $k\in \phibar_1(x_1)\cap \phibar_1(y)$, giving a contradiction.
Let $\varphi_1=\varphi/ P_{x_1}(\eta,k,\varphi)$. Then $k\in \phibar_1(x)$ and $\eta\in \phibar_1(y)$. Let $\varphi_2$ be obtained from $\varphi_1$ by uncoloring $vz, zx_1$ and coloring $zx_1, x_1x_2$ with color $k, 1$, respectively. Then $\eta\in \phibar_2(y)$, $k,\eta\in \phibar_2(v), 1\in \phibar_2(z)$ and $2\in \phibar_2(x_1)$. So $\{v,z,x_1,x_2,y\}$ forms a Kierstead path with respect to $\phiv_2$. Since $x_1,x_2,y$ are $(<\Delta)$-vertices, by Lemma \ref{plong}, the set $\{v,z,x_1,x_2,y\}$ is elementary with respect to $\varphi_2$, this contradicts with the fact that $\eta\in \bar{\varphi_2}(v)\cap \bar{\varphi_2}(y)$.

Then suppose $\eta\in \bar{\varphi}(x_2)\cup \bar{\varphi}(z)$ and
$\bar{\varphi}(v)\cap \bar{\varphi}(x_1)=\emptyset$. Since $d(x_1)<\Delta$, there exists a color $\delta\in \bar{\varphi}(x_1)$. Clearly, $\delta\in \varphi(v)$.
Since $\{z,x_1,x_2\}$ is elementary with respect to $\phiv$, we have $\delta\in \varphi(x_2)\cap \varphi(z)$. Since $v,x_1$ and one of $\{x_2,z\}$ are endvertices of $(\alpha,\beta)$-chains, we claim that $P_{v}(\delta,\eta,\varphi)\neq P_{x_1}(\delta,\eta,\varphi)$. For otherwise, $P_{z}(\delta,\eta,\varphi)\neq P_{x_1}(\delta,\eta,\varphi)$ or $P_{x_2}(\delta,\eta,\varphi)\neq P_{x_1}(\delta,\eta,\varphi)$, let $\phiv_1=\phiv/P_{x_1}(\delta,\eta,\varphi)$, then $\phibar_1(x_1)\cap \phibar_1(z)\neq \emptyset$ or $\phibar_1(x_1)\cap \phibar_1(x_2)\neq \emptyset$, this contradicts with the fact that $\{x_1,x_2,z\}$ is elementary with respect to $\phiv_1$.
Let $\varphi_1=\varphi/ P_{x_1}(\eta,\delta,\varphi)$. Then $\eta\in \bar{\varphi}(x_1)\cap \bar{\varphi}(v)$, a contradiction.

We now show that
\begin{equation}\label{eq4}
\mbox{\ $\bar{\varphi}(v)\cap\bar{\varphi}(v^\prime)=\emptyset$ \ for each\  $v,v^\prime\in Z$.\ }
\end{equation}
If not, let $\alpha\in \bar{\varphi}(v)\cap\bar{\varphi}(v^\prime)$. Since $d(x_1)<\Delta$, we may assume that $\beta\in \bar{\varphi}(x_1)$. Then by (\ref{eq3}), $\beta\in \varphi(v)\cap\varphi(v^\prime)$ and $\alpha\in \phiv(x_1)$.
Then there exists a vertex in $\{v,v^\prime\}$ not in the path $P_{x_1}(\alpha,\beta,\varphi)$. Let $\varphi_1=\varphi/ P_{x_1}(\alpha,\beta,\varphi)$. Then $\alpha\in \bar{\varphi}(v)\cap\bar{\varphi}(x_1)$ or $\alpha\in \bar{\varphi}(v^\prime)\cap\bar{\varphi}(x_1)$, this contradicts with (\ref{eq3}).

By (\ref{eq3}) and (\ref{eq4}), we have $\{x_1,x_2,z\}\cup Z$ is elementary with respect to $\varphi$.
By the symmetry of $y$ and $z$, we have $\{x_1,x_2,y\}\cup Y$ is elementary with respect to $\varphi$.
\end{proof}

By Claim A1, $\{x_1,x_2,z\}\cup Z$ and $\{x_1,x_2,y\}\cup Y$ are both elementary with respect to $\varphi$. Consequently, by the definitions of $C_y$ and $C_z$, we have
$(\Delta-q)|C_z|+|\phibar(x_1)|+|\phibar(x_2)|+|\phibar(z)|<\sum_{v\in\{x_1,x_2,z\}\cup Z}\phibar(v)\le \Delta$ and $(\Delta-q)|C_y|+|\phibar(x_1)|+|\phibar(x_2)|+|\phibar(y)|<\sum_{v\in\{x_1,x_2,y\}\cup Y}\phibar(v)\le \Delta$. It follows that $|C_z|< \frac{d(x_1)+d(x_2)-\Delta-2}{\Delta-q}$ and $|C_y|< \frac{d(x_1)+d(x_2)-\Delta-2}{\Delta-q}$. Hence, \textbf{I} holds.

\subsubsection{Proof of {\bf II}.}

Suppose, on the contrary: $|T_0| > \frac{4(d(x_1)+d(x_2)-\Delta)}{\Delta-q}+\frac{8(d(x_1)+d(x_2)-\Delta-2)}{(\Delta-q)^2}$ for every $\phiv\in \mathcal{C}_{zy}$.
Let $T_0=\{k_1,\cdots,k_{|T_0|}\}$ and
$V_{T_0}=\{z_{k}\in N(z)\,:\,k\in T_0\}\cup\{y_{k}\in N(y)\,:\,k\in T_0\}$.
Let $\phiv$ be an optimal feasible coloring.

{\flushleft \bf Claim B1.}
For each $i\in \bar{\varphi } (x_2)\setminus (R\cup \{1\})$ and
$k\in T$, $P_{x_2}(i,k,\varphi )$ contains both $y$ and $z$; and for each $i\in \bar{\varphi } (x_1)\setminus (R\cup \{2\})$ and
$k\in T$,  $P_{x_1}(i,k,\varphi )$ contains both $y$ and $z$.

\begin{proof}
By symmetry, we only show the first part of statement.
Let $u\in \{y,z\}$ and $v\in \{y,z\}\setminus \{u\}$. Suppose $u\notin V(P_{x_2}(i,k,\varphi ))$. Then $P_u(i, k, \phiv)$ is disjoint with $P_{x_2}(i, k, \phiv)$. Let $\varphi^\prime=\varphi / P_u(i,k,\varphi )$.  Since $i, k\ne 1,2$, $\phiv^\prime$ is also feasible. Since colors in $R$ are unchanged and $d(u_k)< q$, we have $C_u(\phiv^\prime) = C_u\cup \{i\}$ and $C_v(\phiv^\prime) \supseteq C_v$, giving a contradiction to
the maximality of $|C_y| + |C_z|$.
Thus $P_{x_2}(i,k,\varphi )$ contains both $y$ and $z$.
\end{proof}

{\flushleft \bf Claim C1.}  If there exist three vertices $u_1,u_2,u_3\in V_{T_0}$ and two colors $\alpha\notin (\phibar(x_1)\cup \phibar(x_2))\setminus R^\prime$, $\beta\in (\phibar(x_1)\cup \phibar(x_2))\setminus R^\prime$ with $\alpha\in \bar{\varphi}(u_1)\cap \bar{\varphi}(u_2)\cap \bar{\varphi}(u_3)$ and $\beta\in \varphi(u_1)\cap \varphi(u_2)\cap \varphi(u_3)$, then there exists a vertex $u\in \{u_1,u_2,u_3\}$ and an optimal feasible coloring $\varphi^\prime=\varphi/ P_{u}(\alpha,\beta,\varphi)$ such that $\beta\in \bar{\varphi^\prime}(u)$.

\begin{proof}
Since $x_1$ and $x_2$ are symmetric, we  assume $\beta\in \phibar(x_1)$. Note that $\phibar(x_1)$, $\phibar(x_2)$, $\phibar(y)$, $\phibar(z)$ are mutually exclusive, we have $\beta\in \phiv(x_2)\cap\phiv(z)\cap \phiv(y)$.
Since $\alpha\notin (\phibar(x_1)\cup \phibar(x_2))\setminus R^\prime$, we have $\alpha\in (\varphi (x_1)\cap\varphi (x_2))\setminus R^\prime$ or $\alpha\in R^\prime$.

We first consider the case of $\alpha\in (\phiv(x_1)\cap \phiv(x_2))\setminus R^\prime$.
Since $\alpha\in \bar{\varphi}(u_1)\cap \bar{\varphi}(u_2)\cap \bar{\varphi}(u_3)$, we may assume that $P_{x_1}(\alpha,\beta,\phiv)$ is disjoint with $P_{u_1}(\alpha,\beta,\phiv)$.
Let $\varphi_1=\varphi/ P_{u_1}(\alpha,\beta,\varphi)$. Since $\alpha, \beta\notin R^\prime$, $\phiv_1$ is still feasible, $C_y(\phiv_1) = C_y$ and $C_z(\phiv_1) = C_z$. So, $\phiv_1$ is also optimal and $\beta\in \bar{\varphi_1 }(u_1)$.

We now suppose $\alpha\in R\setminus \{1,2\}$. By the definition of $R$, there exists a vertex $v\in \{x_1,x_2,y,z\}$ such that $\alpha\in \phibar(v)$.
It is easy to see that $P_{v}(\alpha,\beta,\varphi)=P_{x_1}(\alpha,\beta,\varphi)$ when $v=x_1$. We claim that $P_{v}(\alpha,\beta,\varphi)=P_{x_1}(\alpha,\beta,\varphi)$ when $v\ne x_1$.
Otherwise, let $\phiv_1=\phiv/P_{x_1}(\alpha,\beta,\varphi)$, then $\alpha\in \phibar_1(v)\cap \phibar_1(x_1)$, this  contradicts  that $\{x_1,x_2,y\}$ and $\{x_1,x_2,z\}$ both are elementary with respect to $\phiv_1$.
Thus $V(P_{x_1}(\alpha,\beta,\phiv))\cap \{u_1,u_2,u_3\}=\emptyset$.
Let $w$ be a vertex of $\{y,z\}$ such that $\alpha\in C_w$.
We may assume that $P_{u_1}(\alpha,\beta,\phiv)$ is disjoint with $P_w(\alpha,\beta,\phiv)$.
If $\alpha\notin C_z\cap C_y$, then $\varphi_1=\varphi/ P_{u}(\alpha,\beta,\varphi)$ is the required coloring such that $\beta\in \bar{\varphi_1 }(u_1)$. Now we suppose that $\alpha\in C_z\cap C_y$. Thus
$\alpha\in \phibar(x_1)\cup \phibar(x_2)$.
We claim that $P_{z}(\alpha,\beta,\varphi)=P_{y}(\alpha,\beta,\varphi)$. For otherwise, let $\varphi_1=\varphi/ P_{z}(\alpha,\beta,\varphi)$, we have $\phiv_1(x_1)=\phiv(x_1)$, $\phiv_1(x_2)=\phiv(x_2)$
or
$\phiv_1(x_1)=(\phiv(x_1)\cup \{\beta\})\setminus\{\alpha\}$, $\phiv_1(x_2)=(\phiv(x_2)\cup \{\alpha\})\setminus\{\beta\}$. Clearly,
$Z(\phiv_1)=Z$ and $Y(\phiv_1)=Y$. Thus $C_z(\phiv_1)=(C_z\cup \{\beta\})\setminus \{\alpha\}$ and $C_y(\phiv_1)=C_y$, which implies that
$|C_z(\phiv_1)|+|C_y(\phiv_1)|=|C_z|+|C_y|$ and $|C_z(\phiv_1)\cap C_y(\phiv_1)|=|C_z\cap C_y|-1$, which leads a contradiction to the minimality of $|C_z\cap C_y|$.
Thus
$\varphi_1=\varphi/ P_{u_1}(\alpha,\beta,\varphi)$ is feasible, $C_y(\phiv_1) = C_y$ and $C_z(\phiv_1) = C_z$.
So $\varphi_1$ is also optimal and $\beta\in \bar{\varphi_1 }(u_1)$.

We finally suppose that $\alpha\in \{1,2\}$.
If $\alpha=1$, then $P_{z}(\alpha,\beta,\varphi)=P_{x_1}(\alpha,\beta,\varphi)=P_{x_2}(\alpha,\beta,\varphi)$.
It is easy to see that $P_{y}(\alpha,\beta,\varphi)$ contains at most two vertices of $\{u_1,u_2,u_3\}$.
Thus there still exists one vertex, say $u_1$, such that $V(P_{u_1}(\alpha,\beta,\varphi))\cap\{x_1,x_2,y,z\}=\emptyset$.
Let $\varphi_1=\varphi/ P_{u_1}(\alpha,\beta,\varphi)$. Thus $\varphi_1$ is  an optimal feasible coloring and $\beta\in \bar{\varphi_1 }(u_1)$.

Then $\alpha=2$.  Clearly, $\{x_1,x_2,y,z\}\subseteq V(P_{x_2}(1,2,\varphi))$. Interchange the colors $1$ and $2$ on all edges not in $P_{x_2}(1,2,\varphi)$ to produce a coloring $\varphi_1$. Then $1\notin \varphi_1(u_i)$ for each $i\in\{1,2,3\}$.
Since $G$ is edge-$\D$-critical, we have $P_{x_1}(\beta,1,\varphi_1)=P_{x_2}(\beta,1,\varphi_1)$. Since $\phiv(zx_1)=1$, we have $P_{x_2}(\beta,1,\varphi_1)$ contains $z$. Note that $P_{y}(1,\beta,\varphi_1)$ contains at most two vertices of $\{u_1,u_2,u_3\}$.
Thus there still exists one vertex, say $u_1$, such that $V(P_{u_1}(1,\beta,\varphi_1))\cap\{x_1,x_2,y,z\}=\emptyset$.
Let $\varphi_2=\varphi_1/ P_{u_1}(1,\beta,\varphi_1)$. Thus $\varphi_2$ is feasible and $C_y(\phiv_2)=C_y$ and $C_z(\phiv_2)=C_z$. So $\varphi_2$ is still optimal and $\beta\in \bar{\varphi_2 }(u_1)$.
\end{proof}

Recall that $|T_0| > \frac{4(d(x_1)+d(x_2)-\Delta)}{\Delta-q}+\frac{8(d(x_1)+d(x_2)-\Delta-2)}{(\Delta-q)^2}$. Let $T=\{k_1, \cdots, k_{|T|}\}$ be a subset of $T_0$ with $|T|= \lfloor\frac{4(d(x_1)+d(x_2)-\Delta)}{\Delta-q}+\frac{8(d(x_1)+d(x_2)-\Delta-2)}{(\Delta-q)^2}\rfloor+1$.
Let \begin{eqnarray*}
V_T & = &\{z_{k_1},z_{k_2},\cdots, z_{k_{|T|}}\}\cup\{y_{k_1},y_{k_2},\cdots, y_{k_{|T|}}\},\\
W(\phiv) & = & \{u\in V_T\,:\, \phibar(u)\cap ((\phibar(x_1)\cup \phibar(x_2))\setminus R^\prime)=\emptyset\},\\
M(\phiv) & = & V_T-W(\phiv)=\{u\in V_T \,:\, \phibar(u)\cap((\phibar(x_1)\cup \phibar(x_2))\setminus R^\prime)\neq \emptyset\},\\
E_T & = & \{zz_{k_1}, zz_{k_2},\cdots, zz_{k_{|T|}}, yy_{k_1},
yy_{k_2},\cdots, yy_{k_{|T|}}\},\\
E_W(\phiv) & = & \{e\in E_T \,:\,  \mbox{ $e$ is incident to a vertex in $W(\phiv)$}\},\mbox{ and } \\
E_M(\phiv) & = & E_T-E_W(\phiv)=\{e\in E_T\,:\, \mbox{ $e$  is incident to a vertex
in $M(\phiv)$} \}.
\end{eqnarray*}
And let $W=W(\phiv),M=M(\phiv),E_W=E_W(\phiv)$ and $E_M=E_M(\phiv)$ if the coloring $\phiv$ is clearly referred.

We assume that $|E_M|$ is maximum over all optimal feasible colorings $\phiv$ and all subsets of $T_0$ with order $|T|$. For each $v\in M$, pick a color $\alpha_v\in \phibar(v)\cap ((\phibar(x_1)\cup \phibar(x_2))\setminus R^\prime)$. Let $C_M=\{\alpha_v \,:\, v\in M\}$. Clearly, $|C_M| \le |M|$.  Note that $\{z_{k_1},z_{k_2},\cdots, z_{k_{|T|}}\}\cap\{y_{k_1},y_{k_2},\cdots, y_{k_{|T|}}\}$ may be not empty, we have $\frac{|E_W|}{2}\leq |W|\leq |E_W|$ and $\frac{|E_M|}{2}\leq |M|\leq |E_M|$.

Since $|T|=\lfloor\frac{4(d(x_1)+d(x_2)-\Delta)}{\Delta-q}+\frac{8(d(x_1)+d(x_2)-\Delta-2)}{(\Delta-q)^2}\rfloor+1$,
by \textbf{I}, we have
\begin{equation}\label{ineq 1}
(\D-q)\left\lceil\frac{|T|}{2}\right\rceil\geq 2(d(x_1)+d(x_2)-\Delta)+2|R|
\geq 2(\D-|(\phibar(x_1)\cup \phibar(x_2))\setminus R^\prime|).
\end{equation}
Since $d(x_1)+d(x_2)<\frac{3}{2}\D$ and $d(x_1),d(x_2)$ are integers, we have $d(x_1)+d(x_2)\leq \frac{3\Delta-1}{2}$. Since $q\leq \Delta-10$, by \textbf{I}, we have $|(\phibar(x_1)\cup \phibar(x_2))\setminus R^\prime|>2\Delta-d(x_1)-d(x_2)-\frac{2(d(x_1)+d(x_2)-\Delta-2)}{\Delta-q}\geq \frac{2\Delta}{5}+1$
and $|T|+\frac{2(\Delta-|(\phibar(x_1)\cup \phibar(x_2))\setminus R^\prime|)}{\Delta-q-1}\leq \frac{\Delta-1}{5}+\frac{\Delta-5}{25}+\frac{6\Delta-2}{45}+1<\frac{2\Delta}{5}+1$.
So
\begin{equation}\label{ineq 2}
|(\phibar(x_1)\cup \phibar(x_2))\setminus R^\prime|>|T|+\frac{2(\Delta-|(\phibar(x_1)\cup \phibar(x_2))\setminus R^\prime|)}{\Delta-q-1}.
\end{equation}

{\flushleft \bf Claim D1.} Suppose no color $k\in T$ such that there exist  three distinct colors $i,j,l\in (\phibar(x_1)\cup \phibar(x_2))\setminus R^\prime$ with $i\in  \bar{\varphi }(z_k)$ and $j,l\in  \bar{\varphi }(y_k)$.
Then there exists an optimal feasible coloring such that $|E_M|\geq |T|+\frac{2(\Delta-|(\phibar(x_1)\cup \phibar(x_2))\setminus R^\prime|)}{\Delta-q-1}.$

\begin{proof}
On the contrary, suppose $|E_M|< |T|+\frac{2(\Delta-|(\phibar(x_1)\cup \phibar(x_2))\setminus R^\prime|)}{\Delta-q-1}.$
Since $|C_M|\le |E_M|$, by (\ref{ineq 2}), there exists a color $\beta\in (\phibar(x_1)\cup \phibar(x_2))\setminus (R^\prime\cup C_M)$. Clearly, $\beta\in \phiv(u)$ for each $u\in W$.

First we claim that $|E_M|\geq |T|+1$. If not, then $|E_M|\leq |T|$ and $|E_W|\geq |T|$ since $|E_M|+|E_W|=2|T|$. Thus by (\ref{ineq 1}), we have $\sum_{v\in W}|\bar{\varphi}(v)|> (\Delta-q)\left\lceil\frac{|T|}{2}\right\rceil>
2(\Delta-|(\phibar(x_1)\cup \phibar(x_2))\setminus R^\prime|)$, which implies that
there exist three vertices $u_1,u_2,u_3\in W$ and a color $\alpha\notin(\phibar(x_1)\cup \phibar(x_2))\setminus R^\prime$ such that $\alpha\in \bar{\varphi}(u_1)\cap \bar{\varphi}(u_2)\cap \bar{\varphi}(u_3)$.
We note that $\beta\in \varphi(u_1)\cap \varphi(u_2)\cap \varphi(u_3)$ since $u_1,u_2,u_3\in W$.
Applying Claim C1 with colors $\alpha$ and $\beta$, we find a vertex $u\in \{u_1,u_2,u_3\}$ and an optimal feasible coloring $\varphi^\prime=\varphi/ P_{u}(\alpha,\beta,\varphi)$ such that $\beta\in \bar{\varphi^\prime}(u)$. Since $u\in W$ and $\beta\notin C_M$, we have $|E_M(\phiv^\prime)|>|E_M|$, giving a contradiction to the maximality of $|E_M|$.

We may assume that $|E_M|=|T|+p$, thus $|E_W|=|T|-p$. Since $|E_M|=|T|+p$, there exists a $p$-element set $Y_1=\{y_{k}\, :\,k\in T \mbox{\ and both } z_{k} \mbox{\ and } y_{k}\in M \}$.
Since there does not exist colors $k\in T$ such that there exist $i,j,l\in (\phibar(x_1)\cup \phibar(x_2))\setminus R^\prime$ with $i\in  \bar{\varphi }(z_k)$, $j,l\in  \bar{\varphi }(y_k)$, we have $|((\phibar(x_1)\cup \phibar(x_2))\setminus R^\prime)\cap \bar{\varphi }(y_{k})|=1$ for each $y_{k}\in Y_1$ since $\bar{\varphi }(y_{k})\cap \bar{\varphi }(z_{k})\cap ((\phibar(x_1)\cup \phibar(x_2))\setminus R^\prime )=\emptyset$. Thus by (\ref{ineq 1}), $\sum_{v\in W\cup Y_1}|\bar{\varphi}(v)\setminus ((\phibar(x_1)\cup \phibar(x_2))\setminus R^\prime)|> (\Delta-q)\left\lceil\frac{|T|-p}{2}\right\rceil+(\Delta-q-1)p\geq (\Delta-q)\left\lceil\frac{|T|}{2}\right\rceil+(\frac{\Delta-q}{2}-1)p>
2(\Delta-|(\phibar(x_1)\cup \phibar(x_2))\setminus R^\prime|)$.
Hence there exist three vertices $u_1,u_2,u_3\in W\cup Y_1$ and a color $\alpha\notin (\phibar(x_1)\cup \phibar(x_2))\setminus R^\prime$ such that $\alpha\in \bar{\varphi}(u_1)\cap \bar{\varphi}(u_2)\cap \bar{\varphi}(u_3)$.
Since $\beta\notin C_M$ and $|((\phibar(x_1)\cup \phibar(x_2))\setminus R^\prime)\cap \bar{\varphi }(y_{k})|=1$ for each $y_{k}\in Y_1$, we have $\beta\in \varphi(u_1)\cap \varphi(u_2)\cap \varphi(u_3)$. Thus by Claim C1, there exists a vertex $u\in \{u_1,u_2,u_3\}$ and an optimal feasible coloring $\varphi^\prime=\varphi/ P_{u}(\alpha,\beta,\varphi)$ such that $\beta\in \bar{\varphi^\prime}(u)$. Thus $u\in M(\phiv^\prime)$.
If $u\in W$, then $|E_M(\phiv^\prime)|>|E_M|$, a contradiction. Thus $u\in Y_1$. We may assume that $u=y_{k_a}$. So there exists three distinct colors $i,j,\beta\in (\phibar^\prime(x_1)\cup \phibar^\prime(x_2))\setminus R^\prime$ such that $i\in  \bar{\varphi^\prime}(z_{k_a})$, $j,\beta\in  \bar{\varphi^\prime}(y_{k_a})$, a contradiction.
\end{proof}

{\flushleft \bf Claim E1.} There exist a color $k\in T$ and three distinct colors $i,j,l\in (\phibar(x_1)\cup \phibar(x_2))\setminus R^\prime$ such that $i\in  \bar{\varphi }(z_k)$, $j,l\in  \bar{\varphi }(y_k)$.

\begin{proof}
We first note that if there exist $i,j\in (\phibar(x_1)\cup \phibar(x_2))\setminus R^\prime$ such that $i\in  \bar{\varphi }(z_k)$, $j\in  \bar{\varphi }(y_k)$, then $i\neq j$; for otherwise, by Claim B1, the path $P_{x_2}(i,k,\varphi )$ or $P_{x_1}(i,k,\varphi )$ contains three endvertices, a contradiction.

Suppose that Claim E1 does not hold.
By Claim D1, $|E_M|\geq |T|+\frac{2(\Delta-|(\phibar(x_1)\cup \phibar(x_2))\setminus R^\prime|)}{\Delta-q-1}$.
So there exists a $\lceil\frac{2(\Delta-|(\phibar(x_1)\cup \phibar(x_2))\setminus R^\prime|)}{\Delta-q-1}\rceil$-elements set
$Y_1=\{y_{k}\,:\,k\in T  \mbox{\ and both } z_{k}\mbox{\ and } y_{k}\in M\}$ such that $|((\phibar(x_1)\cup \phibar(x_2))\setminus R^\prime)\cap \bar{\varphi }(y_{k})|=1$ for each $y_{k}\in Y_1$.
Thus we have $\sum_{v\in Y_1}|\bar{\varphi}(v)\setminus ((\phibar(x_1)\cup \phibar(x_2))\setminus R^\prime)|> (\Delta-q-1)|Y_1|= 2(\Delta-|(\phibar(x_1)\cup \phibar(x_2))\setminus R^\prime|)$.
Thus there exist three vertices $u_1,u_2,u_3\in Y_1$ and a color $\alpha\notin (\phibar(x_1)\cup \phibar(x_2))\setminus R^\prime$ such that $\alpha\in \bar{\varphi}(u_1)\cap \bar{\varphi}(u_2)\cap \bar{\varphi}(u_3)$.
For each $y_{k}\in Y_1$, pick a color $\alpha_{k1}\in \phibar(y_{k})\cap (\phibar(x_1)\cup \phibar(x_2))\setminus R^\prime$ and $\alpha_{k2}\in \phibar(z_{k})\cap (\phibar(x_1)\cup \phibar(x_2))\setminus R^\prime$. Let $C_M^\prime=\{\alpha_{kj} \,:\, y_{k}\in Y_1 \mbox{\ and } 1\le j\le 2\}$. Clearly, $|C_M^\prime|=2|Y_1|$.
By (\ref{ineq 2}), we have $|(\phibar(x_1)\cup \phibar(x_2))\setminus R^\prime|>|C_M^\prime|$, thus there exists a color
$\beta\in (\phibar(x_1)\cup \phibar(x_2))\setminus (R^\prime\cup C_M^\prime)$ such that $\beta\in  \varphi(y_{k})$ for all $y_{k}\in Y_1$.
So $\beta\in \varphi(u_1)\cap \varphi(u_2)\cap \varphi(u_3)$.

Applying Claim C1 with colors $\alpha$ and $\beta$, we find a vertex $u\in \{u_1,u_2,u_3\}$ and an optimal feasible coloring $\varphi^\prime=\varphi/ P_{u}(\alpha,\beta,\varphi)$ such that $\beta\in \bar{\varphi^\prime}(u)$. Then $u\in M(\phiv^\prime)$. Hence, there exists a color $k_a\in T$ with $u=y_{k_a}$ and colors $i,j,\beta\in (\phibar^\prime(x_1)\cup \phibar^\prime(x_2))\setminus R^\prime$ such that $i\in  \bar{\varphi^\prime}(z_{k_a})$, $j,\beta\in  \bar{\varphi^\prime}(y_{k_a})$, giving a contradiction.
\end{proof}

Let $k,i,j,l$ be as stated in Claim E1 and assume  $d(x_2)\le d(x_1)$. First we claim that there is an optimal feasible coloring $\phiv$ such that $i,j,l\in \bar{\varphi }(x_2)\setminus (R\cup \{1\})$.
For otherwise, we may assume $i\in \bar{\varphi }(x_1)\setminus (R\cup \{2\})$.
Since $d(x_2)< \frac{3}{4}\D$, there exists a color $\delta\in \bar{\varphi }(x_2)\setminus (R\cup \{1\})$ such that $\delta\notin \{j,l\}$.
Since $G$ is edge-$\D$-critical, we have $P_{x_1}(i,\delta,\varphi)=P_{x_2}(i,\delta,\varphi)$. Let $\varphi_1=\varphi/ P_{x_1}(i,\delta,\varphi)$. Then $\phiv_1$ is feasible and $i\in \bar{\varphi_1 }(x_2)\setminus (R\cup \{1\})$. Since $R$ does not change, $\phiv_1$ is still optimal.
Since $q\le \D-10$ and $q$ is a positive number, we have $\D\ge 11$. Thus $|\bar{\varphi }(x_2)\setminus (R\cup \{1\})|>\frac{1}{4}\D+1-\frac{\D-4}{10}-1\ge 2$. So $|\bar{\varphi }(x_2)\setminus (R\cup \{1\})|\ge 3$. Using the same method above, we can find an optimal feasible coloring $\phiv^*$ such that $i,j,l\in \bar{\varphi^*}(x_2)\setminus (R\cup \{1\})$.

Since $\{x_1,x_2,y\}$ is elementary with respect to $\phiv$ and $i,j,l\in \phibar(x_2)$, we have $i,j,l\in \phiv(x_1)\cap \phiv(y)$. By Claim B1, we have $P_{x_2}(i,k,\phiv), P_{x_2}(j,k,\phiv)$ both contain $y_k$ and $z_k$, thus $j\in \phiv(z_k)$ and $i\in \phiv(y_k)$. If $2\in \phibar(y_k)$, we let $\phiv_1=\phiv$. Otherwise, we let $\phiv_1$ be obtained from $\phiv$ by interchange the color $2$ and $l$ on all edges not in $P_{x_2}(2,l,\varphi )$. Since $G$ is edge-$\D$-critical, we have $P_{x_1}(2,l,\varphi )=P_{x_2}(2,l,\varphi )$. Thus $2\in \phibar_1(y_k)$.

Let $\varphi_2$ be the coloring obtained from $\varphi_1$ by uncoloring $x_2y$ and coloring $x_1x_2$ with color $2$. Then $2\in \bar{\varphi}_2(y)$.

Let $\phiv_3$ be obtained from $\phiv_2$ by assigning $\phiv_3(yy_k)=2$. Now $k$ is missing both $y$ and $y_k$ and $i$ is still missing at $z_k$. Since $G$ is not $\D$-colorable, $P_{x_2}(i,k,\varphi_3 )=P_{y}(i,k,\varphi_3 )$; otherwise, $\phiv_3/P_{y}(i,k,\varphi_3 )$ can be extended to an edge-$\D$-coloring of $G$, giving a contradiction. Furthermore, $z_k,y_k\notin P_{x_2}(i,k,\varphi_3 )$ since either $i$ or $k$ is missing at these two vertices, which additionally shows that $z\notin P_{x_2}(i,k,\varphi_3 )$ since $\phiv_3(zz_k)=k$.

Let $\varphi _4=\varphi _3/ P_{x_2}(i,k,\varphi_3 )$.
We have $k\in \bar{\varphi _4}(x_2), i \in \bar{\varphi _4}(y)\cap\bar{\varphi _4}(z_k)$ and $ j\in \bar{\varphi _4}(x_2)\cap\bar{\varphi _4}(y_k)$.
Since $G$ is not edge-$\D$-colorable, $P_{x_2}(i,j,\varphi _4)=P_y(i,j,\varphi _4)$ which contains neither $z_k$ nor $y_k$.

Let $\varphi _5=\phiv_4/P_{x_2}(i,j,\varphi _4)$. Then $k\in \bar{\varphi _5}(x_2)$ and $j\in \bar{\varphi _5}(y)\cap \bar{\varphi _5}(y_k)$.

Let $\varphi _6$ be obtained from $\varphi _5$ by recoloring $yy_k$ with $j$. Then $2\in \phibar_6(y)$.

Let $\varphi_7$ be the coloring obtained from $\varphi_6$ by uncoloring $x_1x_2$ and coloring $x_2y$ with color $2$.
Then $1,k\in \bar{\varphi _7}(x_2)$, $2\in \phibar_7(x_1)$, $\varphi _7(x_1z)=1$, $\varphi _7(x_2y)=2$ and $\varphi _7(zz_k)=k$.
Thus $\phiv_7$ is feasible. Since $i,j,l,k\notin R$, the colors in $R$ are unchanged during this sequence of re-colorings, so $C_z(\phiv_7)\supseteq C_z$ and $C_y(\phiv_7)\supseteq C_y$.
Since $\phiv_7(zz_k)=k\in \phibar_7(x)$ and $d(z_k) < q$, we have $k =\phiv_7(zz_k)\in C_z(\phiv_7)$. So, $C_z(\phiv_7) \supseteq C_z\cup \{k\}$. We therefore have $|C_y(\phiv_7)| + |C_z(\phiv_7)| \ge |C_y| + |C_z| +1$, giving a contradiction.
So \textbf{II} holds.
\end{proof}



\end{document}